\documentclass[preprint,11pt]{elsarticle}

\newtheorem{theorem}{Theorem}[section]
\newtheorem{lemma}{Lemma}[section]

\newtheorem{corollary}{Corollary}[section]

\newcommand{\ignore}[1]{}{}

\usepackage{amssymb}
\usepackage{amsmath}

\usepackage{color}
\usepackage{soul}
\usepackage[dvipsnames]{xcolor}

\usepackage[colorlinks=true, linkcolor=blue, citecolor=blue]{hyperref}

\def\1{{{\mbox{${\rm{1\negthinspace\negthinspace I}}$}}}}

\newcommand\beq{\begin{equation}}
\newcommand\eeq{\end{equation}}

\usepackage{ifthen}
\usepackage{xkeyval}
\usepackage{todonotes}
\begin{document}

\begin{frontmatter}

\title{Cram\'{e}r-type moderate deviations and precise large deviations for the elephant random walk with random step sizes}
\author{Xiequan Fan, \ \   Shuyuan Liang \ \   and\ \ Dongxue Ma   }
\address{School of Mathematics and Statistics, Northeastern University at Qinhuangdao, Qinhuangdao, China.\\
E-mail: fanxiequan@hotmail.com.}

\begin{abstract}
In this paper, we   focuses on how the distribution of random step sizes affects Cram\'{e}r-type moderate deviations and precise large deviations for the elephant random walk with random step sizes. First, with   sub-exponential moment condition of random step sizes, we establish Cram\'{e}r-type moderate deviation results. Consequently,  moderate deviation principles are obtained. Secondly, for random step sizes with sub-exponential tails, we establish a large deviation principle  on tail probabilities. Finally, for random step sizes with   regularly varying tails, we obtain precise large deviations.
\end{abstract}

\begin{keyword} Elephant random walk, Cram\'{e}r-type moderate deviations, Large deviation principles, Precise large deviations, Regularly varying tails
\vspace{0.3cm}
\MSC primary 60G42;  60F10;   secondary  60G50
\end{keyword}

\end{frontmatter}




\section{Introduction}

The elephant random walk (ERW) is a random walk model defined on the one-dimensional integer lattice $\mathbb{Z}$, first introduced by Sch\"{u}tz and Trimper \cite{schutz2004elephants} in 2004. In contrast to standard random walks, where each step is independent or depends solely on the current state, the next move of the ERW is determined by randomly selecting an earlier step from its entire history. This model was originally proposed to investigate memory effects in random walk processes. It is known that the ERW undergoes a phase transition between diffusive and superdiffusive regimes, governed by the memory parameter $p$.  It is also known that ERW is a non-Markovian random walk, provided that $p\neq 1/2.$

The ERW model can be described as follows.   At time $n = 0$, we set  the walk at original position $T_0 = 0$. When $n\geq1$, the position of the walk at time $n$ is
\begin{equation*}
T_n=T_{n-1}+X_n,
\end{equation*}
where $X_n$ is step of the walk at time $n$. The distributions of steps $(X_n)_{n\geq1}$ are defined as follows.
Let $ p \in [0,1] $ and $ q \in(0,1) $ be two constants. For $n=1$, it holds
\begin{equation*}
\begin{cases}
\mathbf{P}(X_1=1)=q,\\
\mathbf{P}(X_1=-1)=1-q.
\end{cases}
\end{equation*}
Throughout the paper,  we set $q=1/2$ such that $\mathbf{E}X_1=0$, which means the walk is unbiased. At time $ n \geq 2 $, the value of $X_n$ is determined by the following rule. We select an integer uniformly at random from the set  $ \{1, 2, \ldots, n-1\}$,
  and denote it by $ \beta_{n} $. After that, we define the distributions of $(X_n)_{n\geq2}$ as follows:
\[
\begin{cases}
\mathbf{P}( X_{n} = X_{\beta_{n}} \mid X_1, X_2, \ldots, X_{n-1} ) = p, \\
\mathbf{P}( X_{n} = -X_{\beta_{n}} \mid X_1, X_2, \ldots, X_{n-1} ) = 1 - p.
\end{cases}
\]
Thus, the position $T_n$ of the walk at time $ n $ can be rewritten as follows
\begin{equation} \label{ERW}
T_n = \sum_{k=1}^n X_k,
\end{equation}
where $X_{n} = \alpha_{n} X_{\beta_{n}}$ such that  $\alpha_n$ has a Rademacher distribution $ \mathcal{R}(p) $ for all $n\geq2$, and that $ \alpha_n $ is independent of the random variables $ (X_k)_{1\leq k\leq n} $ and $ (\beta_k)_{2\leq k\leq n} $. Moreover, $ (\beta_n)_{n\geq2} $ is also a sequence of independent random variables. By the pioneering work of Sch\"{u}tz and Trimper \cite{schutz2004elephants}, it is known that the ERW is called diffusive, critical and superdiffusive according to $ p \in [0, 3/4)$, $ p = 3/4$ and $ p \in (3/4, 1]$, respectively.

The ERW has recently attracted a lot of attention.  Baur and Bertoin \cite{ERWPTU} have obtained  functional limit theorems for the position $(T_n)_{n\geq1}$ properly normalized. Coletti, Gava and Sch{\"u}tz \cite{CGSCLT,ASIP} have established the central limit theorem (CLT), strong invariance principle and a law of large numbers.  Bercu \cite{BMERW} also obtained the CLT for the ERW with $p \in (0, 3/4]$  via a martingale method. Bertoin \cite{B21}
studied  how memory impacted passages at the origin for the ERW.
Moreover, variations of the ERW also received considerable attention. Bercu and Laulin \cite{BLMERW} introduced the multi-dimensional ERW and  investigated the multivariate CLT.  Laulin \cite{LLAERW} studied the asymptotic analysis of the amnesic ERW. If the ERW has only a restricted memory, Gut and Stadtm{\"u}ller \cite{GSVERW} considered its asymptotic behaviors. Qin \cite{QS} investigated the recurrence and transience of the multidimensional ERW. Nakano \cite{N} introduced ERW with polynomially decaying steps.  Recently, Bercu \cite{BMERWS} obtained the multivariate CLT for the multi-dimensional ERW with stops.

 In this paper, we focus on the ERW with random step sizes. ERW with random step sizes is a generalization of the ERW,
and it was first introduced by Fan and Shao \cite{FSCMD}. In this model, the step sizes of walk are allowed to vary over time. The model can be described as follows. Let $(Z_k)_{k\geq1}$ be a sequence of i.i.d.\ random variables with mathematical expectation $\mathbf{E}Z_1=\nu$,   $\nu\neq0$, and finite variance $\operatorname{Var}(Z_1)=\sigma^2>0$. Assume that $(Z_k)_{k\geq1}$ is independent of $(X_k)_{k\geq1}$. Let $Y_1$ be the position of the ERW with random step sizes at time $n=1$ such that
\begin{equation*}
 Y_1=X_1Z_1.
\end{equation*}
The position of the ERW with random step sizes  at time $n$, instead of \eqref{ERW}, is
\begin{equation*}
S_n=\sum_{k=1}^{n}Y_k,
\end{equation*}
where
\begin{equation*}
Y_{n}=\alpha_nX_{\beta_n}Z_n,\quad n\geq2.
\end{equation*}
 Without loss of generality, we set $\nu=1$. Otherwise we can replace $S_n$ by $S_n/\nu$. When $\sigma=0$, the ERW with random step sizes degenerates to the usual ERW. The ERW with random step sizes has also attracted a lot of interest. According to the different distributions of $Z_1$, the ERW with random step sizes shows different asymptotic properties. Under various metrics, Dedecker \textit{et al.} \cite{DFHRCCLT} have established the optimal rates of convergence in the CLT,   with step sizes having finite moments of order larger than $2$. With various moment conditions, Hu and Fan \cite{HFDI} gave some deviation inequalities for the ERW with random step sizes. With Cram\'{e}r's condition, Fan and Shao \cite{FSCMD} obtained Cram\'{e}r-type moderate deviations for the ERW with random step sizes.

In this paper, we investigate the Cram\'{e}r-type moderate deviations and precise large deviations for the ERW with random step sizes without Cram\'{e}r's conditions. When the random step sizes satisfy a sub-exponential moment condition, the Cram\'{e}r-type moderate deviation results for martingales  cannot be directly applied. To address this issue, we introduce a truncation argument for random step sizes in this work. For the bounded component, we employ the  Cram\'{e}r-type moderate deviations for martingales from Fan and Shao \cite{FSCMD}. As for the unbounded component, we prove its tail probability is negligible. Combining the analyses of these two components, we finally derive the desired Cram\'{e}r-type moderate deviation results.  Consequently, we obtain the moderate deviation principles. In case where the step sizes have sub-exponential tails, we get the large deviation principles for the tail probability of $S_n$.  For random step sizes with regularly varying tails, we establish the  precise large deviations, also known as Nagaev type large deviations, for tail probabilities of $S_n$.

 The paper is organized as follows.  In Section \ref{sec2}, we present our main results, including Cram\'{e}r-type moderate deviations,  moderate and large deviation principles  and precise large deviations. In Section \ref{lemms2},
 we collect some preliminary lemmas used in the proofs of our main results.   In Sections \ref{bscsd} - \ref{bscsd6},
we give the proofs of our main results.

Throughout the paper,  denote  by $c $ and $C$,    probably supplied with some indices,  generic positive constants that do not depend on $n.$   Moreover, its value  may vary from line to line. For two positive sequences $(a_n)_{n\geq 1}$ and $(b_n)_{n\geq 1}$, we adopt the notation  $a_n\gg b_n $, which means $a_n/b_n\rightarrow\infty$ as $n\rightarrow\infty$.
We also adopt $a_n   \asymp  b_n $, provided that there exists a positive  constant $c$ such that $ 1/c \leq a_n/b_n \leq c$ holds for all $n$ large enough. For two positive functions $f(x)$ and $g(x),$ we also make use of the notation $f(x)\sim g(x)$, which means $f(x)/g(x)\rightarrow 1$ as $x\rightarrow\infty$.

\section{Main results}\label{sec2}
\setcounter{equation}{0}

For convenience, for all $n\geq1$ and all $p\in (0, 1]$, we denote
\[
\ \ a_n=\frac{\Gamma(n)\Gamma(2p)}{\Gamma(n+2p-1)} \ \ \ \  \    \textrm{and} \ \   \ \ \    v_n=\sum_{k=1}^n a_k^2.
\]
Clearly, it holds $a_1=1$.

We can deduce the following Cram\'{e}r-type moderate deviations for ERW with random step sizes, provided that random step sizes $(Z_k)_{k\geq1}$ satisfy the  sub-exponential moment condition. Denote 
by $\Phi (x)$ the standard normal distribution function.
\begin{theorem}\label{thm1}
Assume that there exist two constants $0<\alpha<1, c>0 $ such that
\begin{equation}\label{wda}
\mathbf{E} [(Z_1-1)^2\exp{\{c|Z_1-1|^\alpha\}}]<\infty.
\end{equation}
Assume that $Q_n=   n^{\beta} ,$ $\beta \in (0, 1/2].$ Then
the following inequalities hold for $n$ large enough.
\begin{description}
  \item[\textbf{[i]}]  If $p \in (0,  1/2]$,  then  for all $0\leq x =o(\min{\{\sqrt{n}/Q_n, Q_{n}^{\alpha/2}\}}),$
	\begin{eqnarray}\label{fth111}
  \Bigg| \ln \frac{\mathbf{P}\Big(a_n S_n \geq x\sqrt{v_n  + na_n^2 \sigma^2 }\, \Big)}{1-\Phi \left( x\right)} \Bigg|   \leq   c  \Bigg(x^3\frac{Q_n}{\sqrt{n}}   +  (1+x) \frac{Q_n \ln Q_n }{\sqrt{n}}    \Bigg) .
\end{eqnarray}
  \item[\textbf{[ii]}] If $p \in (\textcolor{black}{1/2}, 3/4)$, then  for all $0\leq x  =o( \textcolor{black}{\min\{n^{(3-4p)/2}/Q_n,\  Q_{n}^{\alpha/2}\}} ),$
  \begin{equation}\label{fth2}
\Bigg| \ln\frac{\mathbf{P}\Big(a_n S_n \geq x\sqrt{v_n  + na_n^2 \sigma^2 }\, \Big)}{1-\Phi \left( x\right)} \Bigg| \leq     c \Bigg(x^3\frac{ Q_n}{    n^{ (3-4p) /2 } }  +  (1+x) \frac{ Q_n  \ln Q_n }{   n^{ (3-4p) /2 } } \Bigg)  .
\end{equation}
  \item[\textbf{[iii]}] If $ p=3/4$,  then   for all $0\leq x =o(\sqrt{\ln n}),$
	\begin{equation}\label{fth3}
\Bigg| \ln \frac{\mathbf{P}\Big(a_n S_n \geq x\sqrt{v_n  + na_n^2 \sigma^2 }\, \Big)}{1-\Phi \left( x\right)}  \Bigg| \leq   c \Bigg(\frac{x^3}{\sqrt{ \ln n}}  +  (1+x) \frac{ \ln \ln n  }{\sqrt{\ln n}}  \Bigg)  .
\end{equation}
\end{description}
Moreover, when $a_nS_n$ is replaced by $-a_nS_n$, the  inequalities \eqref{fth111}, \eqref{fth2} and \eqref{fth3} hold also.
\end{theorem}

When $\alpha=1,$   Cram\'{e}r-type moderate deviations for ERW with random step sizes have been established by Fan and Shao \cite{FSCMD}.
In comparison with the result of Fan and Shao \cite{FSCMD}, the admissible ranges of Cram\'{e}r-type moderate deviations established in Theorem \ref{thm1} shrink considerably as a consequence of condition \eqref{wda}.

From points [i] and [ii] of Theorem \ref{thm1},  we obtain the following result about the equivalence to the normal
tail.
\begin{corollary}\label{cor1.1}
Assume the conditions of Theorem \ref{thm1} are satisfied.
The following equalities hold
\begin{equation}\label{o1}
\frac{\mathbf{P}\Bigl( a_n S_n \geq x \sqrt{ v_n + n a_n^2 \sigma^2 } \Bigr)}{1 - \Phi(x)}
= 1 + o(1)
\end{equation}
and
\begin{equation}\label{o2}
\frac{\mathbf{P}\Bigl( a_n S_n \leq -x \sqrt{v_n + n a_n^2 \sigma^2 } \Bigr)}{\Phi(-x)}
= 1 + o(1)
\end{equation}
for any $p\in (0, 3/4)$ and all $0 \leq x = o\bigl(\min\bigl\{n^{\alpha/(6\alpha+4)}, \textcolor{black}{n^{(3-4p)\alpha/(6\alpha+4)}}\bigr\}\bigr)$ as $n\rightarrow \infty$.
\end{corollary}

By Theorem \ref{thm1}, we also have the following moderate deviation principles (MDP) for the ERW with random step sizes with the sub-exponential moment condition.
\begin{corollary}\label{col0}
Assume that the  condition \eqref{wda} is satisfied.
\begin{description}
\item[\textbf{[i]}] Let $(b_n)_{n\geq1}$ be a sequence of real numbers satisfying $b_n\to\infty$ and $b_n/n^{ \textcolor{black}{\alpha/(2\alpha+4)}}\to 0$ as $n\to\infty$. If $p\in(0,\, 1/2 ]$,  then it holds for any Borel set $ B\subset\mathbb{R}$,
    \begin{equation}\label{MDP1}
    \begin{aligned}
    -\inf_{x\in B^o}\frac{x^2}{2} &\leq\liminf_{n\to\infty}\frac{1}{b_n^2}\ln\mathbf{P}\left(\frac{a_nS_n}{b_n\sqrt{v_n+na_n^2\sigma^2}}\in B \right) \\
    &\leq\limsup_{n\to\infty}\frac{1}{b_n^2}\ln\mathbf{P}\left(\frac{a_nS_n}{b_n\sqrt{v_n+na_n^2\sigma^2}}\in B \right)\leq -\inf_{x\in \bar{B}}\frac{x^2}{2},
    \end{aligned}
    \end{equation}
where $B^o$ and $\bar{B}$ denote the interior and the closure of $B$, respectively.

\item[\textbf{[ii]}] Let $(b_n)_{n\geq1}$ be a sequence of real numbers satisfying  $b_n\to\infty$ and $b_n/n^{\textcolor{black}{(3-4p)\alpha/(2\alpha+4)}}\to 0$ as $n\to\infty$. If $p\in(1/2, \, 3/4)$, then \eqref{MDP1} holds.

\item[\textbf{[iii]}] Let $(b_n)_{n\geq1}$ be a sequence of real numbers satisfying  $b_n\to\infty$ and $b_n/\sqrt{\ln n} \to 0$ as $n\to\infty$.  If $p=3/4$, then \eqref{MDP1} holds.

\end{description}
\end{corollary}

When the random step sizes $(Z_k)_{k\geq1}$ have sub-exponential tails, we have the following large deviation principles (LDP).
\begin{theorem}\label{thm3}
Assume that there exist some constants $\alpha\in(0, 1)$  and $\beta>0$ such that
\begin{equation}\label{LDPc}
\lim_{x\rightarrow\infty}\frac{1}{x^\alpha}\ln \mathbf{P}(|Z_1-1|\geq x)=-\beta.
\end{equation}
Then for all $x > 0$ and all $0\leq p  < 1-\alpha/4$, the following equalities hold
    \[
    \lim_{n \to \infty} \frac{1}{n^\alpha} \ln \mathbf{P}\left( \frac{S_n}{n} \geq x \right) = -\beta\, x^\alpha \ \ \ \
\textrm{and} \ \ \ \
    \lim_{n \to \infty} \frac{1}{n^\alpha} \ln \mathbf{P}\left( \frac{S_n}{n} \leq -x \right) = -\beta \,x^\alpha.
    \]
\end{theorem}

More generally, we have the following result.
\begin{corollary}\label{col1} Assume the condition of Theorem \ref{thm3} is satisfied.
For all $p\in[0, 1-\alpha/4)$ and any Borel set $ B \subset\mathbb{R}$, we have
\begin{eqnarray}
-\inf_{x\in B^o} \beta\,|x|^\alpha&=&\liminf_{n\rightarrow\infty}\frac{1}{n^\alpha}\ln\mathbf{P}\left(\frac{S_n}{n}\in B \right) \nonumber \\
&\leq& \limsup_{n\rightarrow\infty}\frac{1}{n^\alpha}\ln\mathbf{P}\left(\frac{S_n}{n}\in B\right)=-\inf_{x\in\bar{B}}\beta\,|x|^\alpha, \nonumber
\end{eqnarray}
where $ B^o$ and $\bar{B}$ denote the interior and the closure of $B$ respectively.
\end{corollary}

Corollary \ref{col1} states that with condition (\ref{LDPc}),   $(S_n/n)$ satisfies an LDP with speed $n^\alpha$ and a good rate function $\beta|x|^\alpha$.

When the ERW with random step sizes satisfy  the  regularly varying tails condition, we have the following  precise large deviations.
\begin{theorem}\label{thm5}
Assume that there exits a constant $\gamma\in(2, \infty)$ such that
\begin{equation}\label{fcdnsa}
\mathbf{P}(Z_1-1\geq x)\sim \theta\frac{\mathcal{L}(x)}{x^\gamma}\ \  \quad \textrm{and} \ \  \quad \mathbf{P}(Z_1-1\leq -x)\sim (1-\theta)\frac{\mathcal{L}(x)}{x^\gamma}
\end{equation}
as $x \rightarrow \infty,$ where  $\theta\in (0, 1)$ is a constant and $\mathcal{L}(x)$ is a slowly varying function.
Then for all $0\leq p < 3/4$, the following equalities hold
\begin{equation*}
\lim_{n\rightarrow\infty}\sup_{x_n\gg \sqrt{n\ln n}}\left|\frac{\mathbf{P}(S_n\geq x_n)}{n\mathbf{P}(|Z_1|\geq x_n)}-\frac{1}{2}\right|=
\lim_{n\rightarrow\infty}\sup_{x_n\gg \sqrt{n\ln n}}\left|\frac{\mathbf{P}(S_n\leq -x_n)}{n\mathbf{P}(|Z_1|\geq x_n)}-\frac{1}{2}\right|=0.
\end{equation*}
\end{theorem}

The condition (\ref{fcdnsa}) is called as  regularly varying tails condition, see  \cite{MWPLDPRVS}.  Denote  by $\zeta_n=\sum_{k=1}^{n}(Z_k-1)$.
With condition (\ref{fcdnsa}), the following equalities have been introduced in \cite{NAVILTLDP,NSVLDP}:
\begin{equation}\label{PLD+}
\lim_{n\rightarrow\infty}\sup_{x_n\gg b_n}\left|\frac{\mathbf{P}(\zeta_n\geq x_n)}{n\mathbf{P}(Z_1-1\geq x_n)}-\theta\right|=0
\end{equation}
and
\begin{equation}\label{PLD-}
\lim_{n\rightarrow\infty}\sup_{x_n\gg b_n}\left|\frac{\mathbf{P}(\zeta_n\leq -x_n)}{n\mathbf{P}(Z_1-1\leq -x_n)}-(1-\theta)\right|=0.
\end{equation}
Results of the type \eqref{PLD+} and \eqref{PLD-} are called precise large deviations.
For a random walk with a subexponential step-size distribution, precise large deviations can be found in Denisov \emph{et al.}\ \cite{DDSLDP}.
For some dependent sequences, precise large deviations may also hold. For instance, Mikosch and Wintenberger \cite{MWPLDPRVS} have obtained the precise large deviations for  dependent regularly varying sequences.

\section{Preliminary lemmas}\label{lemms2}
\setcounter{equation}{0}
In this section, we collect some preliminary lemmas used in the proofs of the main results.
By the Stirling formula
\[
\ln  \Gamma(x)  =(x-\frac12) \ln x -x + \frac12 \ln 2 \pi + O(\frac{1}{x})  \ \ \ \  \ \textrm{as}\ \ x \rightarrow \infty,
\]
 we deduce that for all $  p\in(0, 1]$,
\begin{eqnarray}\label{an}
\lim\limits_{n\to \infty} a_n n^{2p-1} =\Gamma(2p).
\end{eqnarray}
For all $  p\in(0, 3/4)$, we   have
\begin{eqnarray}\label{vn1}
\lim\limits_{n\to \infty} \frac{ v_n}{ n^{3-4p}} = \frac{\Gamma{(2p)}^2}{3-4p}
\end{eqnarray}
and, for $ p=3/4$,  we also have
\begin{eqnarray}\label{vn2}
\lim\limits_{n\to \infty} \frac{ v_n}{  \ln n}    =\frac{\pi}{4}.
\end{eqnarray}
Denote
\begin{equation}\label{angamma}
\gamma _n=1+\frac{2p-1}{n},\ \ \ \ \ \ \ \ \ n\geq 1.
\end{equation}
For all $n\geq2$, it holds
\begin{equation}\label{an1}
a_n= \prod_{i=1}^{n-1}\frac{1}{\gamma_i}.
\end{equation}
For all $ 0 \leq k \leq 2n $, denote by the filtration $ \mathcal{F}_0 = \sigma\{\emptyset, \Omega\} $ and
\[
\mathcal{F}_{k} =
\begin{cases}
\sigma\{ \alpha_{i}, \beta_{i} : 1 \leq i \leq k \}, & 1 \leq k \leq n, \\
\sigma\{ \alpha_{i}, \beta_{i}, Z_{j} : 1 \leq i \leq n, \; 1 \leq j \leq k - n \}, & n < k \leq 2n.
\end{cases}
\]
It is easy to know
\begin{equation*}
S_n=T_n+\sum_{k=1}^{n}\alpha_kX_{\beta_k}(Z_k-1).
\end{equation*}
Therefore, for all $ 0 \leq k \leq 2n $, define the process
\[
M_{n,k} = \mathbf{E}[a_nS_n \mid \mathcal{F}_{k}] =
\begin{cases}
a_{k} T_{k}, & 1 \leq k \leq n, \\
a_{n} T_{n} + a_{n} \sum_{i=1}^{k-n} \alpha_{i} X_{\beta_{i}} (Z_{i} - 1), & n < k \leq 2n,
\end{cases}
\]
where, as a convention, $\alpha_1$ has a Rademacher distribution $\mathcal{R}(1/2)$ and $X_{\beta_1}=1$.
 Then $ (M_{n,k}, \mathcal{F}_{k})_{1 \leq k \leq 2n} $ is  a martingale with additive decomposition
\begin{equation}\label{Mnk}
M_{n,k} = \mathbf{E}[a_nS_n \mid \mathcal{F}_{k}] =
\begin{cases}
\sum_{i=1}^{k} a_{i} \varepsilon_{i}, & 1 \leq k \leq n, \\
\sum_{i=1}^{n} a_{i} \varepsilon_{i} + a_{n} \sum_{i=1}^{k-n} \alpha_{i} X_{\beta_{i}} (Z_{i} - 1), & n < k \leq 2n,
\end{cases}
\end{equation}
where $ \varepsilon_{i} = T_{i} - \gamma_{i-1} T_{i-1} $ and $ \gamma_{0} T_{0} = 0 $. The martingale differences $ \Delta M_{n,k} = M_{n,k} - M_{n,k-1} $ satisfy
\begin{equation}\label{dMnk}
\Delta M_{n,k} = M_{n,k} - M_{n,k-1} =
\begin{cases}
a_{k} \varepsilon_{k}, & 2 \leq k \leq n, \\
a_{n} \alpha_{k-n} X_{\beta_{k-n}} (Z_{k-n} - 1), & n < k \leq 2n,
\end{cases}
\end{equation}
 with $ \Delta M_{n,1} = M_{n,1} $. Denote by $M_{2n}=M_{n, 2n}$, then $$M_{2n}=a_nS_n.$$
 Let $Q_n$ be a non-random number depending on $n$  and $Q_n\geq1$. Consider the following modified ERW with random step sizes:
\begin{equation}\label{tmerw}
\hat{S}_n=\sum_{k=1}^{n}X_k(Z_k-1)\textbf{I}_{\{|Z_k-1|\leq Q_n\}}+T_n.
\end{equation}
Then we have
\begin{equation}\label{Mnk}
\hat{M}_{n,k} =
\begin{cases}
\sum_{i=1}^{k} a_{i} \varepsilon_{i}, & 1 \leq k \leq n, \\
\sum_{i=1}^{n} a_{i} \varepsilon_{i} + a_{n} \sum_{i=1}^{k-n} \alpha_{i} X_{\beta_{i}} (Z_{i} - 1)\textbf{I}_{\{|Z_i-1|\leq Q_n\}}, & n < k \leq 2n
\end{cases}
\end{equation}
and
\begin{equation}
\langle \hat{M}\rangle_{2n}=\sum_{k=1}^{2n}\mathbf{E}(\Delta\hat{M}_{n,k}^{2}|\mathcal{F}_{k-1}),
\end{equation}
where
\begin{equation}
\Delta\hat{M}_{n,k}=
\begin{cases}
a_{k} \varepsilon_{k}, &  1\leq k \leq n, \\
a_{n} \alpha_{k-n} X_{\beta_{k-n}} (Z_{k-n} - 1)\textbf{I}_{\{|Z_{k-n}-1|\leq Q_n\}}, & n < k \leq 2n.
\end{cases}
\end{equation}
To construct a new martingale for the truncation part, $\Delta\hat{M}_{n,k}$ can be decomposed as follows:
\begin{equation}\label{MDM}
\Delta\hat{M}_{n,k}=\Delta\check{M}_{n,k}+\Delta R_{n,k},
\end{equation}
where
\begin{equation}\label{DeltaMnk}
\Delta\check{M}_{n,k}=
\begin{cases}
a_{k} \varepsilon_{k}, & 1 \leq k \leq n, \\
a_{n} \alpha_{k-n} X_{\beta_{k-n}}\left[ (Z_{k-n} - 1)\textbf{I}_{\{|Z_{k-n}-1|\leq Q_n\}} \right. \\
\left.\ \ \ \ \ \ \ \ \ \ \ \  \ \ \ \ \ \  \    -\mathbf{E}[(Z_{k-n} - 1)\textbf{I}_{\{|Z_{k-n}-1|\leq Q_n\}}]\right], &  n < k \leq 2n
\end{cases}
\end{equation}
and
\begin{equation}\label{DeltaRnk}
\Delta R_{n,k}=
\begin{cases}
0, & 1 \leq k \leq n, \\
a_{n} \alpha_{k-n} X_{\beta_{k-n}}\mathbf{E}[(Z_{k-n} - 1)\textbf{I}_{\{|Z_{k-n}-1|\leq Q_n\}}], & n < k \leq 2n.
\end{cases}
\end{equation}
We denote   $$\check{M}_{2n}=\sum_{k=1}^{2n}\Delta\check{M}_{n,k} \ \ \ \  \textrm{and} \ \ \ \   R_{2n}=\sum_{k=1}^{2n}\Delta R_{n,k}.$$
Clearly, $(\Delta\check{M}_{n,k}, \mathcal{F}_{k})_{1\leq k \leq 2n}$ is also a martingale difference sequence. It has the following property.

\begin{lemma}\label{MNK}
  For all $1\leq k\leq 2n$ and all $p\in (0, 3/4]$, we have for $n$ large enough,
\begin{eqnarray*}
\| \Delta \check{M}_{n,k} \|_\infty &\leq&
\begin{cases}
2\,Q_n a_n, & 0<p\leq 1/2, \\
  2\, Q_n , &1/2<p<3/4,\\
  1 , &p=3/4.
\end{cases}
\end{eqnarray*}
\end{lemma}
\noindent
\textit{Proof.}
For $k=1,$ it holds
\begin{equation*}
\| \Delta \check{M}_{n,1} \|_\infty =\lVert a_1\varepsilon_1\rVert_\infty=\lVert T_1  \rVert_\infty=1.
\end{equation*}
For all $2\leq k\leq n$, we have
\begin{equation*}
\lVert a_k\varepsilon_k\rVert_\infty=\lVert a_kT_k-a_{k-1}T_{k-1}\rVert_\infty=\Big\| a_k X_k-a_k\frac{T_{k-1}}{k-1} (2p-1)\Big\|_\infty\leq 2
\max\{p, 1-p\}a_k\leq 2a_k.
\end{equation*}
For all $n<k\leq 2n$, we have
\begin{align*}
& \big\|a_{n} \alpha_{k-n} X_{\beta_{k-n}}\left[ (Z_{k-n} - 1)\textbf{I}_{\{|Z_{k-n}-1|\leq Q_n\}}-\mathbf{E}[(Z_{k-n} - 1)\textbf{I}_{\{|Z_{k-n}-1|\leq Q_n\}}]\right] \big\|_\infty\\
&\quad=a_n \big\|(Z_{k-n} - 1)\textbf{I}_{\{|Z_{k-n}-1|\leq Q_n\}}-\mathbf{E}[(Z_{k-n} - 1)\textbf{I}_{\{|Z_{k-n}-1|\leq Q_n\}}] \big\|_\infty \\
&\quad\leq 2a_nQ_n.
\end{align*}
Therefore, for all $1\leq k\leq 2n$, it holds
\begin{equation*}
\|\Delta \check{M}_{n,k}\|_\infty\leq2\max\big\{a_1, a_2, \ldots ,a_n, Q_na_n\big\}.
\end{equation*}
Since $a_k=\gamma_ka_{k+1}$, it is easy to see that if $0<p\leq1/2$, then the sequence $(a_k)$ increases monotonously, and that if $1/2<p<3/4$, the sequence $(a_k)$ decreases monotonously. Recall that $1\leq Q_n \leq \sqrt{n}$. When $p=3/4$, by (\ref{an}),  $Q_na_n \rightarrow 0$ as $n \rightarrow \infty$. Therefore, it holds for $n$ large enough,
\begin{equation*}
\max_{1\leq k \leq 2n}\| \Delta \check{M}_{n,k} \|_\infty \leq
\begin{cases}
2\,Q_n a_n, & 0<p\leq 1/2, \\
  2 Q_n  , &1/2<p<3/4,\\
  1, & p=3/4.
\end{cases}
\end{equation*}
This completes the proof of Lemma \ref{MNK}. \hfill\qed

By Lemma 7.2 in \cite{FSCMD}, we have the following lemma.
\begin{lemma}\label{lem2}Denote $\check{\sigma}_n^2=\operatorname{Var} ((Z_k-1)\mathbf{I}_{\{|Z_k-1|\leq Q_n\}})$. Let $\langle \check{M}\rangle_{2n}=\sum_{k=1}^{2n}\mathbf{E}(\Delta\check{M}_{n, k}^{2}|\mathcal{F}_{k-1}).$  Then, it holds for all $x>0,$
\begin{eqnarray*}
&&   \mathbf{P}\bigg( \big| \langle \check{M}\rangle_{2n}  -(v_n  + na_n^2 \check{\sigma}_n^2 )    \big|\geq x (v_n  + na_n^2 \check{\sigma}_n^2 )  \bigg)\\
&&\ \ \ \ \ \ \ \ \ \ \ \ \  \ \ \ \ \ \ \ \ \  \leq\   \left\{ \begin{array}{ll}
\displaystyle c_1 \exp\bigg\{-   c_2 x \frac{n a_n^2\check{\sigma}_n^2 +v_n}{  a_n^2}   \bigg\}, \ \ \ \   & \textrm{$0< p < 1/2$,}\\
0,\ \ \ \   & \textrm{$  p = 1/2$,}\\
\displaystyle  c_1 \exp\bigg\{-c_2 (n a_n^2\check{\sigma}_n^2 +v_n) x     \bigg\},  \ \ \ \    & \textrm{$1/2 < p < 3/4$,}\\
\displaystyle c_1\,   \exp\bigg\{- c_2 (n a_n^2\check{\sigma}_n^2 +v_n) x   \bigg\}, \ \ \ \    & \textrm{$  p =3/4$} .
\end{array} \right.
\end{eqnarray*}
\end{lemma}

We also have the absolute error between $\sigma^2$ and $\check{\sigma}_n^2$ as follows.
\begin{lemma}\label{sigma2}
It holds
\begin{equation}
0\leq\sigma^2-\check{\sigma}_n^2\leq 2\exp{\{-Q_{n}^{\alpha}\}}\mathbf{E}[(Z_1-1)^2\exp{\{|Z_1-1|^\alpha\}}].
\end{equation}
\end{lemma}
\noindent
\textit{Proof.}
We have
$$\sigma^2= \operatorname{Var}(Z_1)= \operatorname{Var}(Z_1-1) \geq  \operatorname{Var}\bigl((Z_1-1) \textbf{I}_{\{|Z_1-1| \leq Q_n\}}\bigr)= \check{\sigma}_n^2,$$
which implies $ \sigma^2-\check{\sigma}_n^2 \geq 0$.
We also have
\begin{align*}
\sigma^2 - \check{\sigma}_n^2
&= \mathbf{E}(Z_1-1)^2 - \operatorname{Var}\bigl((Z_1-1) \textbf{I}_{\{|Z_1-1| \leq Q_n\}}\bigr) \\
&= \mathbf{E}(Z_1-1)^2 - \operatorname{Var}\bigl((Z_1-1) - (Z_1-1) \textbf{I}_{\{|Z_1-1| > Q_n\}}\bigr) \\
&= \mathbf{E}(Z_1-1)^2 - \mathbf{E}(Z_1-1)^2 + 2\mathbf{E}(Z_1-1)^2 \textbf{I}_{\{|Z_1-1| > Q_n\}}    -\operatorname{Var}((Z_1-1)\textbf{I}_{\{|Z_1-1|>Q_n\}}) \\
&\leq 2\mathbf{E}(Z_1-1)^2 \textbf{I}_{\{|Z_1-1| > Q_n\}} \\
&\leq 2\exp{\{-Q_n^\alpha\}} \mathbf{E}[(Z_1-1)^2 \exp{\{|Z_1-1|^\alpha\}}].
\end{align*}
This completes the proof of Lemma \ref{sigma2}.
\hfill\qed

In the proof of Theorem \ref{thm1}, we shall make use of a Cram\'{e}r-type moderate deviation result for normalized martingales. Let $(\xi_i, \mathcal{F}_i)_{i=0,\dots,n}$ be a finite sequence of martingale differences defined on a probability space $(\Omega, \mathcal{F}, \mathbf{P})$, where $\xi_0 = 0$, $\{\emptyset, \Omega\} = \mathcal{F}_0 \subseteq \dots \subseteq \mathcal{F}_n \subseteq \mathcal{F}$ are increasing $\sigma$-fields. Here, the martingale differences $(\xi_i)_{i=1,\dots,n}$ are allowed to depend on $n$. Set
\begin{equation*}
U_0=0,\ \ \ \ \ U_k=\sum_{i=1}^{k}\xi_i, \quad k=1,...,n.
\end{equation*}
Denote by $\left\langle U\right\rangle $   the quadratic characteristic of the
martingale $U=(U_k,\mathcal{F}_k)_{k=0,...,n},$ that is
\begin{equation}\label{quad}
\left\langle U\right\rangle _0=0,\ \ \ \ \ \ \ \ \left\langle U\right\rangle _k=\sum_{i=1}^k\mathbf{E}(\xi _i^2|\mathcal{F}
_{i-1}),\ \ \ \ \quad k=1,...,n.
\end{equation}
In the sequel, we will make use of the following conditions:
\begin{description}
\item[(A1)]  There exists   $\epsilon_n \in (0, \frac12] $   such that
\[
|\mathbf{E}(\xi_{i}^{k}  | \mathcal{F}_{i-1})| \leq \frac12 k!\epsilon_n^{k-2} \mathbf{E}(\xi_{i}^2 | \mathcal{F}_{i-1}),\ \ \   k\geq 2\ \textrm{and} \ 1\leq i\leq n.
\]
\item[(A2)]  There exist  a number  $ \delta_n\in (0, \frac12]$  and a positive constant $C$ such that for all $  x >0 ,$
\[
 \mathbf{P}( \left| \left\langle U\right\rangle _n-1\right|  \geq x   )  \leq  C \exp\{  - x  \delta_n^{-2} \}.
\]
\end{description}

The following lemma has been established by Fan and Shao \cite{FSCMD} (see Theorem 2.2 and Remark 2.1 therein), see also \cite{GHLDP} for bounded martingale differences.
\begin{lemma}\label{CMD31}
Assume that the conditions (A1) and (A2) are satisfied.  Then for all $0 \leq  x   =o( \min\{\epsilon_n^{-1}, $ $  \delta_n ^{-1 } \}),$ it holds
\begin{equation}\label{t0ie1}
 \bigg| \ln \frac{\mathbf{P}(U_n/\sqrt{\left\langle U\right\rangle_n}  >x)}{1-\Phi \left( x\right)} \bigg|\leq c  \bigg( x^3  (\epsilon_n  + \delta_n)+   (1+  x )  \big(  \delta_n|\ln \delta_n| + \epsilon_n|\ln \epsilon_n| \big) \bigg).
\end{equation}
Moreover, the same inequality holds true when $U_n/\sqrt{\left\langle U\right\rangle_n} $ is replaced by $-U_n/\sqrt{\left\langle U\right\rangle_n}$.
\end{lemma}

For any $\alpha\in(0, 1)$, set
\begin{equation}
\Upsilon(U_k)=\sum_{i=1}^{k}\mathbf{E}[\xi_{i}^{2}\exp{\{(\xi_{i}^{+})^\alpha\}}|\mathcal{F}_{i-1}], \quad k=1,...,n.
\end{equation}
By Corollary 2.2 in \cite{FGLDIMA}, we have the following lemma.
\begin{lemma}\label{Upsilon}
Assume  there exits a constant $\alpha\in (0, 1)$ such that $\sum_{k=1}^{n}\mathbf{E}[\xi_{k}^{2}\exp{\{(\xi_{k}^{+})^\alpha\}}]<\infty$. If
\begin{equation}
\|\Upsilon(U_n)\|_{\infty}=o(n^{2-\alpha}) \quad as\quad n\rightarrow \infty,
\end{equation}
then for any $x\geq0$, it holds
\begin{equation}
\limsup_{n \to \infty} \frac{1}{n^\alpha} \ln \mathbf{P}\left( \frac{1}{n} \max_{1 \leq k \leq n} U_k \geq x \right) \leq -x^\alpha.
\end{equation}
\end{lemma}
 For $S_n$, we have the following decomposition
\begin{equation}\label{dec}
S_n=T_n +N_n,
\end{equation}
where $T_n=\sum_{k=1}^{n}X_k$ is defined by (\ref{ERW})  and $N_n=\sum_{k=1}^{n}A_k$, with $A_k=X_k(Z_k-1), 1\leq k \leq n$. According to  the proof of Theorem 2.3 in \cite{FGLDIMA}, we can obtain the following lemma.
\begin{lemma}\label{LDPzk-1}
If the condition  \eqref{LDPc} is satisfied, then we have  for all $x \geq 0$,
    \begin{equation}\label{decdfs1}
    \lim_{n \to \infty} \frac{1}{n^\alpha} \ln \mathbf{P}\left(\max_{1 \leq k \leq n} \frac{  N_k}{n} \geq x \right) = \lim_{n \to \infty} \frac{1}{n^\alpha} \ln \mathbf{P}\left( \frac{N_n}{n} \geq x \right) = -\beta\, x^\alpha
    \end{equation}
    and
    \begin{equation}\label{decdfs2}
    \lim_{n \to \infty} \frac{1}{n^\alpha} \ln \mathbf{P}\left(\min_{1 \leq k \leq n} \frac{  N_k}{n} \leq - x \right) = \lim_{n \to \infty} \frac{1}{n^\alpha} \ln \mathbf{P}\left( \frac{N_n}{n} \leq -x \right) = -\beta\, x^\alpha.
    \end{equation}
\end{lemma}
\noindent
\textit{Proof.}
Without loss of generality, we assume that $\beta=1$ in the condition \eqref{LDPc}. Otherwise,  we consider $\beta^{1/\alpha}(Z_k-1)$ instead of $(Z_k-1)$. By the condition  \eqref{LDPc}, for any $\varepsilon \in (0, 1),$ it holds
 $$C_\varepsilon:=\mathbf{E}[((1-\varepsilon)|Z_1-1|)^2\exp{\{((1-\varepsilon)|Z_1-1|)^{\alpha}\}}]<\infty.$$ Thus, it holds $$\|\Upsilon((1-\varepsilon)N_n)\|_{\infty}=\sum_{k=1}^{n}\mathbf{E}[((1-\varepsilon)|Z_1-1|)^2\exp{\{((1-\varepsilon)|Z_1-1|)^{\alpha}\}}]=nC_\varepsilon=o(n^{2-\alpha})$$ as $n\rightarrow\infty$. Therefore, by Lemma \ref{Upsilon} with $U_k=(1-\varepsilon)N_k$, we have  for all $x\geq 0$,
\begin{equation}\nonumber
\limsup_{n \to \infty} \frac{1}{n^\alpha} \ln \mathbf{P}\left( \frac{1}{n} \max_{1 \leq k \leq n} N_k \geq  \frac{x}{1-\varepsilon} \right) \leq -x^\alpha,
\end{equation}
that is, for all $x\geq 0$,
\begin{equation}\nonumber
\limsup_{n \to \infty} \frac{1}{n^\alpha} \ln \mathbf{P}\left( \frac{1}{n} \max_{1 \leq k \leq n} N_k \geq x \right) \leq -x^\alpha (1-\varepsilon)^\alpha.
\end{equation}
Letting $\varepsilon\rightarrow 0,$ we get for all $x\geq0$,
\begin{equation}\label{uperbound}
\limsup_{n \to \infty} \frac{1}{n^\alpha} \ln \mathbf{P}\left( \frac{1}{n} \max_{1 \leq k \leq n} N_k \geq x \right) \leq -x^\alpha.
\end{equation}

Next, we give a proof for the lower bound.  It is easy to see that for all $x \geq 0$ and any $\varepsilon>0$,
\begin{eqnarray}
 \mathbf{P}(N_n\geq nx) &\geq& \mathbf{P}\big(N_n-N_1\geq-n\varepsilon, N_1\geq n(x+\varepsilon) \big) \nonumber \\
& =&\mathbf{P}\big(N_n-N_1\geq-n\varepsilon, N_1\geq n(x+\varepsilon), X_1=1 \big) \nonumber \\
& &\ \  + \mathbf{P}\big(N_n-N_1\geq-n\varepsilon, N_1\geq n(x+\varepsilon), X_1=-1 \big) \nonumber \\
& =& \mathbf{P}\big(N_n-N_1\geq-n\varepsilon, Z_1-1\geq n(x+\varepsilon), X_1=1 \big) \nonumber \\
& &\ \   + \mathbf{P}\big(N_n-N_1\geq-n\varepsilon, Z_1-1\leq -n(x+\varepsilon), X_1=-1 \big) \nonumber \\
&=& \mathbf{P}\big(N_n-N_1\geq-n\varepsilon, X_1=1)\mathbf{P}(Z_1-1\geq n(x+\varepsilon)\big) \nonumber \\
& &\ \   + \mathbf{P}\big(N_n-N_1\geq-n\varepsilon, X_1=-1)\mathbf{P}(Z_1-1\leq -n(x+\varepsilon)\big) \nonumber \\
& =& \frac{1}{2}\mathbf{P}\big(N_n-N_1\geq-n\varepsilon| X_1=1)\mathbf{P}(Z_1-1\geq n(x+\varepsilon)\big) \nonumber \\
&&\ \   + \frac{1}{2}\mathbf{P}\big(N_n-N_1\geq-n\varepsilon| X_1=-1)\mathbf{P}(Z_1-1\leq -n(x+\varepsilon)\big). \label{fsddsggs}
\end{eqnarray}
Clearly,  $(A_k, \mathcal{F}_{n+k})_{1\leq k \leq n}$ is a finite sequence of martingale differences defined on a probability space $(\Omega ,\mathcal{F},\mathbf{P})$. Notice that $\sum_{k=2}^{n}\mathbf{E}[A_k^2|\mathcal{F}_{n+k-1}]=\sum_{k=2}^{n}\mathbf{E}[(Z_k-1)^2]=(n-1)\sigma^2$ is a non-random number. For any $\varepsilon>0$, it holds
\begin{align*}
&\frac{1}{(n-1)\sigma^2} \sum_{k=2}^{n} \mathbf{E}\left[ A_k^2 \mathbf{I}_{\{|A_k| > \varepsilon \sqrt{(n-1)\sigma^2}\}} \mid \mathcal{F}_{n+k-1} \right] \\
& = \frac{1}{(n-1)\sigma^2} \sum_{k=2}^{n} \mathbf{E}\left[ (Z_k - 1)^2 \mathbf{I}_{\{|Z_k - 1| > \varepsilon \sqrt{(n-1)\sigma^2}\}}\right] \\
&  \leq \frac{1}{(n-1)\sigma^2}\sum_{k=2}^{n}\exp{\{-\varepsilon^\alpha\sigma^\alpha(n-1)^{\alpha/2}\}}\mathbf{E}[(Z_k-1)^2\exp{\{|Z_k-1|^\alpha\}}]\rightarrow0,\ \  n\rightarrow\infty.
\end{align*}
Therefore,  conditional Lindeberg's condition for martingales is satisfied, and then we have the following central limit theorem:
\begin{equation}\label{normal}
\frac{N_n-N_1}{\sqrt{n-1}\sigma} \xrightarrow{d}\mathcal{N}(0, 1),
\end{equation}
where $\xrightarrow{d}$ stands for convergence in distribution.
Moreover, $(N_n)_{n\geq1}$ is a martingale such that $N_n/n\to0$ in probability. Of course \eqref{normal} implies this convergence $\mathbf{P}(N_n-N_1\geq-n\varepsilon| X_1=1)\to1$. And for $\mathbf{P}(N_n-N_1\geq-n\varepsilon| X_1=-1)$, we have the same conclusion. Notice that $$\frac{1}{2}\mathbf{P}\Big(Z_1-1\geq n(x+\varepsilon)\Big)+\frac{1}{2}\mathbf{P}\Big(Z_1-1\leq -n(x+\varepsilon)\Big)=\frac{1}{2}\mathbf{P}\Big(|Z_1 - 1| \geq n(x + \varepsilon)\Big).$$ By  the condition  \eqref{LDPc}, we have the following lower bound
\begin{equation*}
\frac{1}{2}\mathbf{P}\Big(|Z_1 - 1| \geq n(x + \varepsilon) \Big) \geq \frac{1}{2}\exp \Big \{ -\left(n(x + \varepsilon)\right)^\alpha(1 + \varepsilon) \Big  \}
\end{equation*}
for sufficiently large $n$. Therefore, it follows from (\ref{fsddsggs}) that for any $x\geq 0,$
\begin{equation*}
\liminf_{n\rightarrow\infty}\frac{1}{n^\alpha}\ln \mathbf{P}\bigg(\frac{N_n}{n}\geq x\bigg)\geq -(x+\varepsilon)^\alpha(1+\varepsilon).
\end{equation*}
Letting $\varepsilon\rightarrow 0$, we obtain  the lower bound: For any $x\geq0,$
\begin{equation}\label{lowebound}
\liminf_{n\rightarrow\infty}\frac{1}{n^\alpha}\ln \mathbf{P}\bigg(\frac{N_n}{n}\geq x\bigg)\geq -x^\alpha.
\end{equation}
Clearly, it holds
  \begin{equation} \nonumber
  \ln \mathbf{P}\left(\max_{1 \leq k \leq n} \frac{  N_k}{n} \geq x \right) \geq   \ln \mathbf{P}\left( \frac{N_n}{n} \geq x \right) .
    \end{equation}
By the last inequality and the inequalities (\ref{lowebound}) and (\ref{uperbound}), we get (\ref{decdfs1}). The inequalities (\ref{decdfs2})
can be proved in a similar method. This completes the proof of lemma. \hfill\qed

\section{Proofs of Theorem \ref{thm1}  and its Corollaries }\label{bscsd}
\setcounter{equation}{0}

\subsection{Proof of Theorem \ref{thm1}}
 Denote by
\begin{equation*}
\frac{\check{M}_{2n}}{\sqrt{v_n+na_{n}^{2}\check{\sigma}_n^2}}=\sum_{k=1}^{2n}\xi_k,
\end{equation*}
where
\begin{equation*}
\xi_k=\frac{\Delta\check{M}_{n,k}}{\sqrt{v_n+na_{n}^{2}\check{\sigma}_n^2}}.
\end{equation*}
By the proof of Lemma \ref{MNK}, we have
\begin{equation}\label{epsilon}
\|\xi_k\|_\infty \leq \epsilon_n :=
\begin{cases}
\displaystyle \frac{2\,Q_n a_n}{\sqrt{v_n+na_{n}^{2}\check{\sigma}_n^2}}, & 0<p\leq 1/2, \\
\displaystyle \frac{ 2 Q_n }{\sqrt{v_n+na_{n}^{2}\check{\sigma}_n^2}} , &1/2<p<3/4,\\
\displaystyle \frac{ 1}{\sqrt{v_n+na_{n}^{2}\check{\sigma}_n^2}}, & p=3/4.
\end{cases}
\end{equation}
From (\ref{an})-(\ref{vn2}), we can deduce that
\begin{eqnarray}\label{ineq37}\textcolor{black}{
\displaystyle
\epsilon_n   \asymp  \left\{ \begin{array}{ll}\displaystyle
  \ \frac{Q_n }{\sqrt{n}}, \ \   & \textrm{ $0<p \leq  1/2$},\\
 \displaystyle  \frac{Q_n }{ n^{  (3-4p)/2}}, \ \   & \textrm{ $1/2 <p<3/4$}, \\
\displaystyle \frac{1}{\sqrt{\ln n}}  ,  \ \   & \textrm{ $  p=3/4$.}
\end{array} \right. }
\end{eqnarray}
Moreover, by Lemma \ref{lem2}, we have for all $x>0$,
\begin{eqnarray*}
  \mathbf{P}\Bigg( \Big|\sum_{k=1}^{2n}  \mathbf{E}( \xi_k^2| \mathcal{F}_{k-1}) -1   \Big |\geq x   \Bigg)
&\leq& \left\{ \begin{array}{ll}
\displaystyle c_1 \exp\bigg\{-   c_2 x \frac{v_n + na_n^2 \check{\sigma}_n^2}{  a_n^2}   \bigg\}, \ \    & \textrm{$0< p < 1/2$},\\
0,\ \ \ \   & \textrm{$  p = 1/2$},\\
\displaystyle  c_1 \exp\bigg\{-c_2  ( v_n + na_n^2 \check{\sigma}_n^2)x     \bigg\},  \ \    & \textrm{$1/2 < p < 3/4$},\\
\displaystyle c_1\,   \exp\bigg\{- c_2 (v_n + na_n^2 \check{\sigma}_n^2) x   \bigg\}, \ \    & \textrm{$  p =3/4$} .
\end{array} \right.
\end{eqnarray*}
Using relations \eqref{an}-\eqref{vn2}, we have that for all $x>0$ and for $n$ large enough,
\begin{eqnarray*}
  \mathbf{P}\bigg( \Big|\sum_{k=1}^{2n} \mathbf{E}( \xi_k^2| \mathcal{F}_{k-1}) -1    \Big|\geq x   \bigg)
&\leq&  c_1\,   \exp\bigg\{- c_2\, \delta_n^{-2} x   \bigg\},
\end{eqnarray*}
where $\delta_n$ satisfies
\begin{eqnarray}\nonumber
\displaystyle
\delta_n   \asymp  \left\{ \begin{array}{ll}\displaystyle
  \ n^{-1/2}, \ \   & \textrm{ $0<p \leq1/2 , $ }\\
      \ n^{- (3-4p)/2}, \ \   & \textrm{ $1/2 < p <3/4,  $ } \\
\displaystyle    ( \ln n)^{-1/2},  \ \   & \textrm{ $  p=3/4$.}
\end{array} \right.
\end{eqnarray}
Applying Lemma \ref{CMD31} to $\frac{\check{M}_{2n}}{\sqrt{v_n+na_{n}^{2}\check{\sigma}_n^2}}$, we obtain the following results for $n$ large enough.
\begin{description}
  \item[\textbf{[i]}]  If $p \in (0,  1/2]$,  then  for all $0\leq x =o(\sqrt{n}/Q_n),$
	\begin{equation}\label{fth1}
\Bigg| \ln \frac{\mathbf{P}\big(\check{M}_{2n} \geq x\sqrt{v_n  + na_n^2 \check{\sigma}_n^2 }\, \big)}{1-\Phi \left( x\right)} \Bigg| \leq   c   \Bigg(x^3\frac{Q_n}{\sqrt{n}}   +  (1+x) \frac{Q_n \ln Q_n }{\sqrt{n}}  \Bigg) .
\end{equation}
  \item[\textbf{[ii]}] If $p \in (1/2, 3/4)$, then   for all $0\leq x  =o(n^{(3-4p)/2}/Q_n),$
  \begin{equation}\nonumber
\Bigg| \ln\frac{\mathbf{P}\big(\check{M}_{2n} \geq x\sqrt{v_n  + na_n^2 \check{\sigma}_n^2 }\, \big)}{1-\Phi \left( x\right)} \Bigg| \leq     c  \Bigg(x^3\frac{Q_n }{    n^{ (3-4p) /2 } }  +  (1+x) \frac{Q_n  \ln n }{   n^{ (3-4p) /2 } } \Bigg)  .
\end{equation}
  \item[\textbf{[iii]}] If $ p=3/4$,  then   for all $0\leq x =o(\sqrt{\ln n}),$
	\begin{equation} \nonumber
\Bigg| \ln \frac{\mathbf{P}\big(\check{M}_{2n} \geq x\sqrt{v_n  + na_n^2 \check{\sigma}_n^2 }\, \big)}{1-\Phi \left( x\right)}  \Bigg| \leq   c \Bigg(\frac{x^3}{\sqrt{ \ln n}}  +  (1+x) \frac{ \ln \ln n  }{\sqrt{\ln n}}  \Bigg)  .
\end{equation}
\end{description}

We first consider  the case $p\in (0,\, 1/2]$.
As $\hat{M}_{2n}=\check{M}_{2n}+R_{2n}$, by \eqref{fth1}, we have for all $0\leq x=o(\sqrt{n}/Q_n)$,
\begin{equation}\label{fth11}
\Bigg| \ln \frac{\mathbf{P}\Big(\hat{M}_{2n}-R_{2n} \geq x\sqrt{v_n  + na_n^2 \check{\sigma}_n^2 }\, \Big)}{1-\Phi \left( x\right)} \Bigg| \leq   c  \Bigg(x^3\frac{Q_n}{\sqrt{n}}   +  (1+x) \frac{Q_n \ln Q_n }{\sqrt{n}}  \Bigg).
\end{equation}
As $ \mathbf{E}(Z_k-1)=0$, it holds
\begin{align}
|R_{2n}|&=\bigg|\sum_{k=1}^{n}a_n\alpha_kX_{\beta_k} \mathbf{E}[(Z_k-1)\textbf{I}_{\{|Z_k-1|>Q_n\}}]\bigg| \nonumber \\
&\leq a_n\sum_{k=1}^{n}\big| \mathbf{E}[(Z_k-1)\textbf{I}_{\{|Z_k-1|>Q_n\}}]\big|. \nonumber
\end{align}
By Markov's inequality, we have
\begin{equation}
| \mathbf{E}[(Z_k-1)\textbf{I}_{\{|Z_k-1|>Q_n\}}]|\leq \mathbf{E}[(Z_1-1)^2\exp{\{c|Z_1-1|^\alpha\}}] \exp\{  -c\,Q_{n}^{\alpha} \}.
\end{equation}
Then we  obtain
\begin{equation}\label{Rn}
|R_{2n}|\leq \varepsilon_n := na_n\exp\{-c\,Q_{n}^{\alpha}\} \mathbf{E}[(Z_1-1)^2\exp{\{c|Z_1-1|^\alpha\}}].
\end{equation}
For $\mathbf{P}\Big(\hat{M}_{2n}\geq x\sqrt{v_n  + na_n^2 \check{\sigma}_n^2}\Big) $, it holds for all $  \varepsilon_n/\sqrt{v_n  + na_n^2 \check{\sigma}_n^2 }\leq x
=o(\sqrt{n}/Q_n)$ and $n$ large enough,
\begin{align}\label{MDP+PP}
&\quad\left|\ln\frac{\mathbf{P}\Big(\hat{M}_{2n}\geq x\sqrt{v_n  + na_n^2 \check{\sigma}_n^2 }\Big)}{1-\Phi(x)}\right|\notag\\
&\leq\left|\ln\frac{\mathbf{P}\Big(\hat{M}_{2n}-R_{2n} \geq x\sqrt{v_n  + na_n^2 \check{\sigma}_n^2 }-\varepsilon_n\,\Big)}{1-\Phi\big(x-\varepsilon_n/\sqrt{v_n  + na_n^2 \check{\sigma}_n^2 }\big)}\right| + \left|\ln\frac{1-\Phi\big(x-\varepsilon_n/\sqrt{v_n  + na_n^2 \check{\sigma}_n^2 }\big)}{1-\Phi(x)}\right|  \nonumber \\
&=: I_1+I_2. \nonumber
\end{align}
By \eqref{fth11}, we have the following upper bound for $I_1$: For all $  \varepsilon_n/\sqrt{v_n  + na_n^2 \check{\sigma}_n^2 }\leq x
=o(\sqrt{n}/Q_n)$ and $n$ large enough,
\begin{eqnarray}\label{MDP+PP1}
I_1
&\leq& c\Bigg(\bigg(x- \frac{\varepsilon_n}{ \sqrt{v_n+na_n^2\check{\sigma}_n^2}}\bigg)^3\frac{Q_n}{\sqrt{n}}+\bigg(1+ \bigg(x-\frac{\varepsilon_n}{\sqrt{v_n + na_n^2\check{\sigma}_n^2}} \bigg)\bigg)\frac{Q_n\ln Q_n}{\sqrt{n}}  \Bigg)\notag\\
&\leq& c\Bigg(x^3\frac{Q_n}{\sqrt{n}}+(1+x)\frac{Q_n\ln Q_n}{\sqrt{n}}\Bigg) .
\end{eqnarray}
It is known that
\begin{equation}\label{phi}
\frac{1}{\sqrt{2\pi}(1 + x)}e^{-x^2/2} \leq 1 - \Phi(x) \leq \frac{1}{\sqrt{\pi}(1 + x)} e^{-x^2/2}, \quad x \geq 0
\end{equation}
(cf. \cite{FSCMD}), which implies that
\begin{equation}\label{phi^-1}
\frac{1+x}{\sqrt{2}}\leq\frac{\phi(x)}{1-\Phi(x)}\leq1+x,
\end{equation}
where $\phi(x)=(2\pi)^{-1/2}\exp\{-x^2/2\}$ is the density function of standard normal distribution.
By  \eqref{phi^-1},  we have for all $0\leq x=o(\sqrt{n}/Q_n)$ and $n$ large enough,
\begin{align}\label{phitail}
  I_2 \leq c (1+x)\varepsilon_n/\sqrt{v_n  + na_n^2 \check{\sigma}_n^2 }.
\end{align}
By \eqref{MDP+PP1}, \eqref{phitail} and the fact that $Q_n=n^{\beta}$, we deduce that for all $ \varepsilon_n/\sqrt{v_n  + na_n^2 \check{\sigma}_n^2 }\leq  x=o(\sqrt{n}/Q_n)$,
\begin{eqnarray}\label{fth12}
\Bigg| \ln \frac{\mathbf{P}\Big(\hat{M}_{2n} \geq x\sqrt{v_n  + na_n^2 \check{\sigma}_n^2 }\, \Big)}{1-\Phi \left( x\right)} \Bigg| &\leq&   c  \Bigg(x^3\frac{Q_n}{\sqrt{n}}   +  (1+x)\left( \frac{Q_n \ln Q_n }{\sqrt{n}}+  \frac{\varepsilon_n }{\sqrt{v_n  + na_n^2 \check{\sigma}_n^2 }} \right)  \Bigg)\nonumber\\
&\leq&   c  \Bigg(x^3\frac{Q_n}{\sqrt{n}}   +  (1+x) \frac{Q_n \ln Q_n }{\sqrt{n}}    \Bigg)\nonumber .
\end{eqnarray}
 When $\check{M}_{2n} $ is replaced by $-\check{M}_{2n} $,   (\ref{fth1}) holds also. Thus,
 following an argument similar to the proof of Corollary 2.3 in Fan and Shao \cite{FSCMD},
  we have the following Berry-Esseen bound:
\begin{equation*}
\sup_{x }
\Biggl|\mathbf{P}\left(\check{M}_{2n}\geq x\sqrt{v_n+na_n^2\check{\sigma}_n^2}\right)-\Big(1-\Phi(x)\Big)\Biggr|
\leq c\frac{Q_n\ln Q_n}{\sqrt{n}},
\end{equation*}
which implies that
\begin{align}\label{BEM}
 \sup_{x }\left|\mathbf{P}\left(\hat{M}_{2n}\geq x\sqrt{v_n+na_n^2\check{\sigma}_n^2}\right)-\Big(1-\Phi(x)\Big)\right| \leq c\frac{Q_n\ln Q_n}{\sqrt{n}}.
\end{align}
By \eqref{BEM} and \eqref{phi}, we have   for all   $0\leq x \leq \varepsilon_n/\sqrt{v_n  + na_n^2 \check{\sigma}_n^2 }$,
\begin{align*}
\left|\ln\frac{\mathbf{P}\left(\hat{M}_{2n}\geq x\sqrt{v_n+na_n^2\check{\sigma}_n^2}\right)}{1-\Phi(x)}\right|
\leq c\frac{Q_n\ln Q_n}{\sqrt{n}}.
\end{align*}
Therefore, it holds for all $0\leq x=o(\sqrt{n}/Q_n)$,
\begin{equation*}
\left|\ln\frac{\mathbf{P}\left(\hat{M}_{2n}\geq x\sqrt{v_n+na_n^2\check{\sigma}_n^2}\right)}{1-\Phi(x)}\right|
\leq c  \Bigg(x^3\frac{Q_n}{\sqrt{n}}   +  (1+x) \frac{Q_n \ln Q_n }{\sqrt{n}}    \Bigg).
\end{equation*}
By Lemma \ref{sigma2}, we obtain   for all $0\leq x=o(\sqrt{n}/Q_n)$,
\begin{equation}\label{fth13}
\Bigg| \ln \frac{\mathbf{P}\Big(\hat{M}_{2n} \geq x\sqrt{v_n  + na_n^2 \sigma^2 }\, \Big)}{1-\Phi \left( x\right)} \Bigg| \leq   c  \Bigg(x^3\frac{Q_n}{\sqrt{n}}   +  (1+x) \frac{Q_n \ln Q_n }{\sqrt{n}}   \Bigg) .
\end{equation}
Clearly, we have for all $0\leq x =o(\min{\{\sqrt{n}/Q_n, Q_{n}^{\alpha/2}\}})$ and $n$ large enough,
\begin{align}\label{M^hat}
& \left| \ln \frac{\mathbf{P}\Bigl(M_{2n}    \geq x\sqrt{v_n + na_n^2 \sigma^2}\Bigr)}{1-\Phi(x)} \right| \notag\\
& = \left| \ln \frac{\mathbf{P}\Bigl(\hat{M}_{2n}    \geq x\sqrt{v_n + na_n^2 \sigma^2}\Bigr)}{1-\Phi(x)} \cdot \frac{\mathbf{P}\Bigl( M_{2n}    \geq x\sqrt{v_n + na_n^2 \sigma^2}\Bigr)}{\mathbf{P}\Bigl(\hat{M}_{2n}    \geq x\sqrt{v_n + na_n^2 \sigma^2}\Bigr)}  \right| \notag\\
& \leq \left| \ln \frac{\mathbf{P}\Bigl(\hat{M}_{2n}    \geq x\sqrt{v_n + na_n^2 \sigma^2}\Bigr)}{1-\Phi(x)} \right|+\left|\ln \frac{\mathbf{P}\Bigl( M_{2n}\geq x\sqrt{v_n + na_n^2 \sigma^2}\Bigr)}{\mathbf{P}\Bigl(\hat{M}_{2n}    \geq x\sqrt{v_n + na_n^2 \sigma^2}\Bigr)} \right|.
\end{align}
Moreover, notice that
\begin{align*}
& \bigg| \mathbf{P}\left( M_{2n}>x\sqrt{v_n+na_n\sigma^2} \right) - \mathbf{P}\left( \hat{M}_{2n}>x\sqrt{v_n+na_n\sigma^2} \right) \bigg| \\
&\quad \ \ \ \ \leq\ \mathbf{P}\left( \bigcup_{k=1}^n \{ |Z_k-1| > Q_n \} \right) \leq \sum_{k=1}^n \mathbf{P}\left( |Z_k-1| > Q_n  \right) \\
&\quad \ \ \ \ \  \leq \ n Q_n^{-2} \exp{\{-c\, Q_n^\alpha\}}  \mathbf{E} [(Z_1-1)^2 \exp\{ c |Z_1-1|^\alpha \}].
\end{align*}
For the second term in the bound (\ref{M^hat}), we have for all $0\leq x =o(\min{\{\sqrt{n}/Q_n, Q_{n}^{\alpha/2}\}})$ and $n$ large enough,
\begin{align*}
\frac{\mathbf{P}\Bigl( M_{2n}\geq x\sqrt{v_n + na_n^2 \sigma^2}\Bigr)}{\mathbf{P}\Bigl(\hat{M}_{2n}\geq x\sqrt{v_n + na_n^2 \sigma^2}\Bigr)}
&\leq1+\frac{\left|\mathbf{P}\left( M_{2n}>x\sqrt{v_n+na_n\sigma^2}\right)-\mathbf{P}\left(\hat{M}_{2n}>x\sqrt{v_n+na_n\sigma^2} \right) \right|}{\mathbf{P}\Bigl(\hat{M}_{2n}\geq x\sqrt{v_n + na_n^2 \sigma^2}\Bigr)}\notag\\
&\leq1+\frac{C_1n Q_n^{-2} \exp{\{-c\, Q_n^\alpha\}}}{\mathbf{P}\Bigl(\hat{M}_{2n}\geq x\sqrt{v_n + na_n^2 \sigma^2}\Bigr)}\notag\\
&=1+\frac{C_1n Q_n^{-2} \exp{\{-c\, Q_n^\alpha\}}}{1-\Phi(x)}\cdot\frac{1-\Phi(x)}{\mathbf{P}\Bigl(\hat{M}_{2n}\geq x\sqrt{v_n + na_n^2 \sigma^2}\Bigr)}\notag\\
&\leq1+C_2(1+x)n Q_n^{-2} \exp{\left\{\frac{1}{2}x^2-c\, Q_n^\alpha\right\}},
\end{align*}
which implies the upper bound:
\begin{equation}\label{upperM^hat}
\ln \frac{\mathbf{P}\Bigl( M_{2n}\geq x\sqrt{v_n + na_n^2 \sigma^2}\Bigr)}{\mathbf{P}\Bigl(\hat{M}_{2n}\geq x\sqrt{v_n + na_n^2 \sigma^2}\Bigr)}
\leq C_2(1+x)n Q_n^{-2} \exp{\left\{\frac{1}{2}x^2-c\, Q_n^\alpha\right\}}.
\end{equation}
By a similar argument, it holds for all $0\leq x =o(\min{\{\sqrt{n}/Q_n, Q_{n}^{\alpha/2}\}})$ and $n$ large enough,
\begin{equation}\label{lnlower}
\frac{\mathbf{P}\Bigl( M_{2n}\geq x\sqrt{v_n + na_n^2 \sigma^2}\Bigr)}{\mathbf{P}\Bigl(\hat{M}_{2n}\geq x\sqrt{v_n + na_n^2 \sigma^2}\Bigr)}
\geq 1-C_3(1+x)n Q_n^{-2} \exp{\left\{\frac{1}{2}x^2-c\, Q_n^\alpha\right\}}.
\end{equation}
By \eqref{lnlower} and the fact that $\ln(1-y)\geq-2y$, where $y\in[0, 1/2)$, we have the lower bound: For all $0\leq x =o(\min{\{\sqrt{n}/Q_n, Q_{n}^{\alpha/2}\}})$ and $n$ large enough,
\begin{equation}\label{lowerM^hat}
\ln \frac{\mathbf{P}\Bigl( M_{2n}\geq x\sqrt{v_n + na_n^2 \sigma^2}\Bigr)}{\mathbf{P}\Bigl(\hat{M}_{2n}\geq x\sqrt{v_n + na_n^2 \sigma^2}\Bigr)}
\geq-C_4(1+x)n Q_n^{-2} \exp{\left\{\frac{1}{2}x^2-c\, Q_n^\alpha\right\}}.
\end{equation}
By \eqref{M^hat}, \eqref{upperM^hat} and \eqref{lowerM^hat}, we have for all $0\leq x =o(\min{\{\sqrt{n}/Q_n, Q_{n}^{\alpha/2}\}})$ and $n$ large enough,
\begin{align*}
\left| \ln \frac{\mathbf{P}\Bigl(M_{2n}    \geq x\sqrt{v_n + na_n^2 \sigma^2}\Bigr)}{1-\Phi(x)} \right|
&\leq c \Bigg( x^3 \frac{Q_n}{\sqrt{n}} + (1+x) \left( \frac{Q_n \ln Q_n }{\sqrt{n}} + nQ_n^{-2} \exp\left\{\frac{1}{2}x^2 - c_1\,Q_{n}^{\alpha}\right\} \right)   \Bigg)\\
& \leq c \Bigg( x^3 \frac{Q_n}{\sqrt{n}} + (1+x)  \frac{Q_n \ln Q_n }{\sqrt{n}}  \Bigg).
\end{align*}
Recall $M_{2n}=a_n S_n.$ Thus, we obtain for all $0\leq x =o(\min{\{\sqrt{n}/Q_n, Q_{n}^{\alpha/2}\}})$ and $n$ large enough,
\begin{align*}
 \left| \ln \frac{\mathbf{P}\Bigl(a_n S_n \geq x\sqrt{v_n + na_n^2 \sigma^2}\Bigr)}{1-\Phi(x)} \right|
   \leq  c \Bigg( x^3 \frac{Q_n}{\sqrt{n}} + (1+x)  \frac{Q_n \ln Q_n }{\sqrt{n}}   \Bigg),
\end{align*}
which gives the first desired inequality for the case $p\in (0, 1/2]$. The remaining cases can be proved in a similar way.

When  $a_nS_n$ is replaced by $-a_nS_n$,  the assertion of Theorem \ref{thm1} follows by a similar argument.
This completes the proof of Theorem \ref{thm1}.

\subsection{Proof of Corollary \ref{cor1.1}}
For the case $p\in (0,1/2]$, if $0\leq x= \min\{\sqrt{n}/Q_n, Q_n^{\alpha/2}\}$, $x^3Q_n/\sqrt{n}=o(1)$ as soon as $\beta=1/(3\alpha+2)$.
Taking $\beta=1/(3\alpha+2)$, from (\ref{fth111}), we obtain for all $0\leq x=o(n^{\alpha/(6\alpha+4)})$,
\begin{equation}
\frac{\mathbf{P}\big(a_nS_n>x\sqrt{v_n+na_n\sigma^2} \,\big)}{1-\Phi(x)}=1+o(1).
\end{equation}
When $a_nS_n$ is replaced by $-a_nS_n$,  (\ref{fth111}) holds also. Therefore,  the last equality
holds also when $a_nS_n$ is replaced by $-a_nS_n$.
For $p \in (1/2, 3/4),$ the assertion follows by (\ref{fth2}). \hfill\qed

\subsection{Proof of Corollary \ref{col0}}

 We only provide the proof for the case $p\in(0,\, 1/2]$, and the remaining proof follows by a similar argument.
  First, we give the proof for the upper bound: For  any Borel set $ B\subset\mathbb{R}$,
\begin{equation}\label{up1}
\limsup_{n\to\infty}\frac{1}{b_n^2}\ln\mathbf{P}\left(\frac{a_nS_n}{b_n\sqrt{v_n+na_n^2\sigma^2}\ }\in  B \right)\leq -\inf_{x\in \bar{ B }}\frac{x^2}{2}.
\end{equation}
When $ B =\emptyset$, equality \eqref{up1} holds trivially. For any given non-empty Borel set $ B \subset \mathbb{R}$, denote by $x_0=\inf_{x\in  B  }|x|$. Since $ B \subset\bar{ B }$, we have $x_0\geq \inf_{x\in \bar{ B }}|x|$ and thus it holds $x_0^2/2\geq \inf_{x\in \bar{ B }}x^2/2$.  For the case $p\in (0,1/2]$, if $0\leq x=\min\{\sqrt{n}/Q_n, Q_n^{\alpha/2}\}$, $xQ_n/\sqrt{n}=o(1)$ as soon as $\beta=1/( \alpha+2)$. By Theorem \ref{thm1} with $\beta=1/( \alpha+2)$, the following inequalities hold
\begin{align*}
&\quad\mathbf{P}\left(\frac{a_nS_n}{b_n\sqrt{v_n+na_n^2\sigma^2}}\in  B \right) \leq \mathbf{P}\left(\left|\frac{a_nS_n}{\sqrt{v_n+na_n^2\sigma^2}}\right|\geq b_nx_0\right) \\
&\ \ \ \ \ \ \ \ \ \ \ \ \ \  \leq 2\Big(1- \Phi(b_nx_0)\Big)\exp\left\{c\left(\frac{(b_nx_0)^3}{n^{ \alpha/(2\alpha+4) }}+(1+b_nx_0)\frac{\ln n}{n^{ \alpha/(2\alpha+4)}}\right)\right\}.
\end{align*}
 Since that  $b_n\to\infty$ and $b_n/n^{ \alpha/(2\alpha+4)}\to 0$ as $n\to\infty$, we have
\begin{equation}\label{lnexp}
\frac{1}{b_n^2}\ln\exp\left\{c \left(\frac{(b_nx_0)^3}{n^{ \alpha/(2\alpha+4)}}+(1+b_nx_0)\frac{\ln n}{n^{ \alpha/(2\alpha+4)}}\right)\right\}=o(1),\quad  \quad n\rightarrow\infty.
\end{equation}
By \eqref{phi} and \eqref{lnexp}, it holds
\[
\limsup_{n\to\infty}\frac{1}{b_n^2}\ln\mathbf{P}\left(\frac{a_nS_n}{b_n\sqrt{v_n+na_n^2\sigma^2}}\in  B \right)\leq -\frac{x_0^2}{2}\leq -\inf_{x\in \bar{ B }}\frac{x^2}{2}.
\]
This completes the proof of (\ref{up1}).

Next, we present a proof for  the lower bound
\begin{equation}\label{low1}
\liminf_{n\to\infty}\frac{1}{b_n^2}\ln\mathbf{P}\left(\frac{a_nS_n}{b_n\sqrt{v_n+na_n^2\sigma^2}}\in  B \right)\geq -\inf_{x\in B ^o}\frac{x^2}{2}.
\end{equation}
When  $ B ^o=\emptyset$, \eqref{low1} holds trivially. For any $ B ^o\neq\emptyset$, since $ B ^o$ is an open set,
 for any $\varepsilon_1>0$, there exists $x_1\in  B ^o$ such that
\begin{equation}\label{epsilon1}
0<\frac{x_1^2}{2}\leq \inf_{x\in  B ^o}\frac{x^2}{2}+\varepsilon_1.
\end{equation}
Without loss of generality, we assume that $x_1>0$. For $\varepsilon_2\in(0,x_1)$ sufficiently small, we have $[x_1-\varepsilon_2,x_1+\varepsilon_2)\subset  B ^o\subset  B $. Then it holds
\begin{align*}
&\quad\mathbf{P}\left(\frac{a_nS_n}{b_n\sqrt{v_n+na_n^2\sigma^2}}\in  B \right) \\
& \geq \mathbf{P}\left(\frac{a_nS_n}{\sqrt{v_n+na_n^2\sigma^2}}\in \left[b_n(x_1-\varepsilon_2),b_n(x_1+\varepsilon_2)\right)\right) \\
& = \mathbf{P}\left(\frac{a_nS_n}{\sqrt{v_n+na_n^2\sigma^2}}\geq b_n(x_1-\varepsilon_2)\right) - \mathbf{P}\left(\frac{a_nS_n}{\sqrt{v_n+na_n^2\sigma^2}}\geq b_n(x_1+\varepsilon_2)\right).
\end{align*}
By Theorem \ref{thm1} with $\beta=1/( \alpha+2)$   and \eqref{phi}, the ratio satisfies
\begin{align*}
&\quad \frac{\mathbf{P}\bigg(\dfrac{a_nS_n}{\sqrt{v_n+na_n^2\sigma^2}}\geq b_n(x_1+\varepsilon_2)\bigg)}{\mathbf{P}\bigg(\dfrac{a_nS_n}{\sqrt{v_n+na_n^2\sigma^2}}\geq b_n(x_1-\varepsilon_2)\bigg)}  \leq  \frac{ \displaystyle \ \ \frac{1}{\sqrt{\pi}\big(1+b_n(x_1+\varepsilon_2) \big)} e^{-b_n^2(x_1+\varepsilon_2)^2/2 }\Big(1+o(1) \Big)\ \ }{\displaystyle \ \ \frac{1}{\sqrt{2 \pi}\big(1+ b_n(x_1-\varepsilon_2)\big)}  e^{-b_n^2(x_1-\varepsilon_2)^2/2 }\Big(1+o(1) \Big)\ \ } \\
&\ \ \ \   \leq\sqrt{2} \Big(1+ b_n(x_1-\varepsilon_2) \Big) \exp\Big\{-2b_n^2x_1\varepsilon_2 \Big\}\Big(1+o(1) \Big) \to 0,\ \  \quad   n\to\infty,
\end{align*}
which means $\mathbf{P}\Big(\frac{a_nS_n}{\sqrt{v_n+na_n^2\sigma^2}}\geq b_n(x_1+\varepsilon_2)\Big)$ converges to $0$ faster than $\mathbf{P}\Big(\frac{a_nS_n}{\sqrt{v_n+na_n^2\sigma^2}}\geq b_n(x_1-\varepsilon_2)\Big)$.
Thus for $n$ large enough, it holds
\begin{align*}
&\quad\mathbf{P}\left(\frac{a_nS_n}{b_n\sqrt{v_n+na_n^2\sigma^2}}\in  B \right) \geq \frac{1}{2}\mathbf{P}\left(\frac{a_nS_n}{\sqrt{v_n+na_n^2\sigma^2}}\geq b_n(x_1-\varepsilon_2)\right) \\
&\ \ \ \ \ \ \  \ \ \ \ \ \ \ \ \ \ \ \ \  \ \ \ \ \geq \frac{1}{2\sqrt{2}(1+b_n(x_1-\varepsilon_2))}\exp\left\{-\frac{b_n^2(x_1-\varepsilon_2)^2}{2}+o(1)\right\}.
\end{align*}
Taking logarithms and scaling, we can deduce that
\[
\liminf_{n\to\infty}\frac{1}{b_n^2}\ln\mathbf{P}\left(\frac{a_nS_n}{b_n\sqrt{v_n+na_n^2\sigma^2}}\in  B \right)\geq -\frac{(x_1-\varepsilon_2)^2}{2}.
\]
Letting $\varepsilon_2\to 0$, by  \eqref{epsilon1}, we obtain
\[
\liminf_{n\to\infty}\frac{1}{b_n^2}\ln\mathbf{P}\left(\frac{a_nS_n}{b_n\sqrt{v_n+na_n^2\sigma^2}}\in  B \right)\geq -\frac{x_1^2}{2}\geq -\inf_{x\in  B ^o}\frac{x^2}{2}-\varepsilon_1.
\]
 Since $\varepsilon_1$ can be arbitrarily small, the lower bound \eqref{low1} follows.

 Combining \eqref{up1} and \eqref{low1} together, we obtain the desired inequalities  for the case $p\in(0,\,  1/2 ]$.

\section{Proofs of Theorem \ref{thm3} and its Corollary }
\setcounter{equation}{0}
\subsection{Proof of Theorem \ref{thm3}}
 We first consider the tail probability of $T_n/n$. By the inequalities  (3.13), (3.16) and (3.17) in \cite{HFDI}, we can obtain the following deviation inequality:
For all $p\in [0, 1)$, $x>0$ and $n$ large enough,
\begin{equation}\label{LDPT}
\mathbf{P}\left(\frac{T_n}{n}\geq x\right)\leq
\begin{cases}
\exp\{-c_1nx^2\},&\quad p\in[0, 3/4),\\
\exp\{-c_2nx^2/\ln n\},&\quad p=3/4,\\
\exp\{-c_2n^{4-4p} x^2 \},&\quad p\in(3/4, 1).\\
\end{cases}
\end{equation}
 Next, we consider the upper bound of $\ln\mathbf{P}(S_n\geq nx)$: For all $x>0$ and all $0<\varepsilon<x$,
\begin{align}
\ln \mathbf{P}(S_n\geq nx) &\leq \ln \Big[\mathbf{P}(N_n\geq n(x-\varepsilon))+\mathbf{P}(T_n\geq n\varepsilon) \Big] \nonumber  \\
                           &=\ln\left[1+\frac{\mathbf{P}(T_n\geq n\varepsilon)}{\mathbf{P}(N_n\geq n(x-\varepsilon))}\right]+\ln\Big[\mathbf{P}(N_n\geq n(x-\varepsilon))\Big].   \label{LDPdsT1}
\end{align}
When $0 \leq p < 1-\alpha/4$,  by Lemma \ref{LDPzk-1}  and \eqref{LDPT},   $\mathbf{P}(T_n\geq n\varepsilon)$  converges to $0$ faster than  $\mathbf{P}(N_n\geq n(x-\varepsilon))$
as $n\rightarrow \infty$.
Thus it holds
\begin{equation*}
\lim_{n\rightarrow\infty}\ln\left[1+\frac{\mathbf{P}(T_n\geq n\varepsilon)}{\mathbf{P}(N_n\geq n(x-\varepsilon))}\right]=0.
\end{equation*}
Therefore, the first term in the right hand side  of (\ref{LDPdsT1}) converges to 0 as $n\rightarrow\infty$. Therefore, we have
\begin{equation}\label{upSn}
\limsup_{n\rightarrow\infty}\frac{1}{n^\alpha}\ln\mathbf{P}\left(\frac{S_n}{n}\geq x\right)\leq \lim_{n\rightarrow\infty}\frac{1}{n^\alpha}\ln\mathbf{P}\bigg(\frac{N_n}{n}\geq x-\varepsilon \bigg)=-\beta(x-\varepsilon)^\alpha.
\end{equation}
Letting $\varepsilon\rightarrow 0$, we obtain   for all $x>0$,
\begin{equation}\label{vcgd1}
\limsup_{n\rightarrow\infty}\frac{1}{n^\alpha}\ln\mathbf{P}\left(\frac{S_n}{n}\geq x\right)\leq -\beta x^\alpha.
\end{equation}

For the lower bound, we have for any $x, \varepsilon>0$,
\begin{align*}
\mathbf{P}(S_n\geq nx) &\geq \mathbf{P}(N_n\geq n(x+\varepsilon), T_n\geq -n\varepsilon) \\
                       &= \mathbf{P}(N_n\geq n(x+\varepsilon))- \mathbf{P}(N_n\geq n(x+\varepsilon), T_n< -n\varepsilon) \\
                       &\geq \mathbf{P}(N_n\geq n(x+\varepsilon))- \mathbf{P}(T_n< -n\varepsilon).
\end{align*}
From the last line, we get for any $x, \varepsilon>0$,
\begin{equation*}
\frac{\mathbf{P}(S_n\geq nx)}{\mathbf{P}(N_n\geq n(x+\varepsilon))} \geq 1-\frac{\mathbf{P}(T_n< -n\varepsilon)}{\mathbf{P}(N_n\geq n(x+\varepsilon))}.
\end{equation*}
Denote by
\begin{equation*}
\mu_n=\frac{\mathbf{P}(T_n<- n\varepsilon)}{\mathbf{P}(N_n\geq n(x+\varepsilon))}.
\end{equation*}
Then it holds for any $x, \varepsilon>0$,
\begin{equation*}
\mathbf{P}(S_n\geq nx) \geq (1-\mu_n)\mathbf{P}(N_n\geq n(x+\varepsilon)).
\end{equation*}
When $0 \leq p < 1-\alpha/4$, again by Lemma \ref{LDPzk-1}  and \eqref{LDPT}, it is easy to know that $\mathbf{P}(T_n<- n\varepsilon)$  converges to $0$ faster than $\mathbf{P}(N_n\geq n(x+\varepsilon))$ as $n \rightarrow \infty$.
Thus, we have for $0 \leq p < 1-\alpha/4$,
\begin{equation*}
\lim_{n\rightarrow\infty}\mu_n=0.
\end{equation*}
Therefore, by Lemma \ref{LDPzk-1},  we have for any $x, \varepsilon>0$,
\begin{equation}\label{epsilon}
\liminf_{n\rightarrow\infty}\frac{1}{n^\alpha}\ln \mathbf{P}\left(\frac{S_n}{n}\geq x\right)\geq - \beta(x+\varepsilon)^\alpha.
\end{equation}
Letting $\varepsilon \rightarrow 0$, we obtain the following inequality: For any $x>0$,
\begin{equation}\label{vcgd2}
\liminf_{n\rightarrow\infty}\frac{1}{n^\alpha}\ln \mathbf{P}\left(\frac{S_n}{n}\geq x\right)\geq -\beta x^\alpha.
\end{equation}
Combining (\ref{vcgd1}) and (\ref{vcgd2}) together, we obtain the first desired equality.
The proof for $\lim_{n\rightarrow\infty}\frac{1}{n^\alpha}\ln\mathbf{P}(\frac{S_n}{n}\leq -x)$ is similar. This completes the proof of Theorem \ref{thm3}.
\subsection{Proof of Corollary \ref{col1}}
 First, we give a proof for the upper bound
\begin{equation}\label{upper}
\limsup_{n\rightarrow\infty}\frac{1}{n^\alpha}\ln\mathbf{P}\left(\frac{S_n}{n}\in  B \right)\leq-\inf_{x\in\bar{ B }}\beta |x|^\alpha.
\end{equation}
When $ B =\emptyset$, \eqref{upper} holds trivially. For any non-empty Borel set $ B \subset\mathbb{R}$, denote by $x_0=\inf_{x\in B }|x|$. Since $ B \subset\bar{ B }$, it is easy to know that $x_0\geq\inf_{x\in\bar{ B }}|x|$ and thus $-\beta x_0^\alpha\leq-\inf_{x\in\bar{ B }}\beta |x|^\alpha$. By Theorem \ref{thm3}, we have
\begin{equation}\label{up}
\limsup_{n\rightarrow\infty}\frac{1}{n^\alpha}\ln\mathbf{P}\left(\frac{S_n}{n}\in B \right)\leq \limsup_{n\rightarrow\infty}\frac{1}{n^\alpha}\ln\mathbf{P} \left(\left|\frac{S_n}{n}\right|\geq x_0\right)
= -\beta x_0^\alpha\leq-\inf_{x\in\bar{ B }}\beta |x|^\alpha ,
\end{equation}
which gives  (\ref{upper}).

Next, we give a proof for the lower bound
\begin{equation}\label{lower}
\liminf_{n\rightarrow\infty}\frac{1}{n^\alpha}\ln\mathbb{P}\left(\frac{S_n}{n}\in B \right)\geq-\inf_{x\in B ^o}\beta |x|^\alpha.
\end{equation}
Clearly, inequality \eqref{lower} holds trivially when $ B ^o=\emptyset$. Therefore, we only need to consider the case $ B ^o\neq\emptyset$. For any $\varepsilon>0,$ there exists a  $x_0 \in  B ^o,$ such that
\begin{equation}
\inf_{x\in B ^o}\beta |x|^\alpha+\varepsilon  \geq \beta |x_0|^\alpha >0 .
\end{equation}
Without loss of generality, we assume that $x_0>0$. For any $\delta > 0$ such that $[x_0 - \delta, x_0 + \delta) \subset  B ^o$ and $x_0-\delta>0$, we have

\begin{equation*}
\mathbf{P}\left( \frac{S_n}{n} \in  B  \right) \geq \mathbf{P}\left( \frac{S_n}{n} \in [x_0 - \delta, x_0 + \delta) \right).
\end{equation*}
Notice that
\[
\mathbf{P}\left( \frac{S_n}{n} \in [x_0 - \delta, x_0 + \delta) \right) = \mathbf{P}\left( \frac{S_n}{n} \geq x_0 - \delta \right) - \mathbf{P}\left( \frac{S_n}{n} \geq x_0 + \delta \right).
\]
By Theorem \ref{thm3}, we get
\[
\lim_{n \to \infty} \frac{1}{n^\alpha} \ln \mathbf{P}\left( \frac{S_n}{n} \geq x_0 - \delta \right) = -\beta (x_0 - \delta)^\alpha
\]
and
\[
\lim_{n \to \infty} \frac{1}{n^\alpha} \ln \mathbf{P}\left( \frac{S_n}{n} \geq x_0 + \delta \right) = -\beta (x_0 + \delta)^\alpha.
\]
Since $-\beta (x_0 + \delta)^\alpha < -\beta (x_0 - \delta)^\alpha$, the probability $\mathbf{P}\left( \frac{S_n}{n} \geq x_0 + \delta \right)$ decays at a faster exponential rate than $\mathbf{P}\left( \frac{S_n}{n} \geq x_0 - \delta \right)$. Therefore, for large $n$, we have

\[
\mathbf{P}\left( \frac{S_n}{n} \in [x_0 - \delta, x_0 + \delta) \right) \geq \frac{1}{2} \mathbf{P}\left( \frac{S_n}{n} \geq x_0 - \delta \right).
\]
Thus, it holds
\[
\liminf_{n \to \infty} \frac{1}{n^\alpha} \ln \mathbf{P}\left( \frac{S_n}{n} \in  B  \right) \geq \liminf_{n \to \infty} \frac{1}{n^\alpha} \ln \mathbf{P}\left( \frac{S_n}{n} \in [x_0 - \delta, x_0 + \delta) \right) = -\beta (x_0 - \delta)^\alpha.
\]
Since $\delta$ can be  arbitrarily small, we deduce that
 \[
\liminf_{n \to \infty} \frac{1}{n^\alpha} \ln \mathbf{P}\left( \frac{S_n}{n} \in  B  \right) \geq    -\beta  x_0 ^\alpha \geq -\inf_{x\in B ^o}\beta |x|^\alpha-\varepsilon .
\]
Now,  as $\varepsilon$  can be arbitrarily small,  we obtain
\begin{equation}\nonumber
\liminf_{n \to \infty} \frac{1}{n^\alpha} \ln \mathbf{P}\left( \frac{S_n}{n} \in  B ^o \right) \geq -\inf_{x \in  B ^o} \beta |x|^\alpha.
\end{equation}
which gives (\ref{lower}).

Combining \eqref{upper} and \eqref{lower} together, we obtain the desired inequalities. This completes the proof of Corollary \ref{col1}.

\section{Proof of Theorem \ref{thm5}}\label{bscsd6}
\setcounter{equation}{0}
We only give the proof for the case $\mathbf{P}(S_n\geq x_n)$, with $x_n\gg \sqrt{n\ln n},$ and the proof for the case $\mathbf{P}(S_n\leq -x_n)$ follows by  a similar argument.   Similar to the proof of Theorem \ref{thm3}, we begin the proof with $\mathbf{P}(N_n\geq x_n)$. It holds
\begin{align*}
\mathbf{P}(N_n\geq x_n) &= \sum_{k=1}^{n}\mathbf{P}\left(N_n\geq x_n, A_k>y_n, \max_{i\neq k}A_i\leq y_n\right)\\
                      &\quad \ + \mathbf{P}\bigg(N_n\geq x_n, \bigcup_{1\leq i<j\leq n}\{A_i>y_n, A_j>y_n\}\bigg)\\
                      &\quad \ + \mathbf{P}\left(N_n\geq x_n, \max_{1\leq k\leq n}A_k\leq y_n\right)\\
                      &:= P_{1n}+P_{2n}+P_{3n}.
\end{align*}
Next, we give some estimations for $P_{1n}, P_{2n}$ and $ P_{3n}$.
In the sequel, for simplicity of notation, we denote $y_n\equiv x_n/\lambda_n$, where $\lambda_n=\ln n$.\\

\textbf{(I)} \textbf{Control of $P_{1n}$}. It is easy to see that
\begin{align*}
\mathbf{P}(N_n\geq x_n, A_k>y_n, \max_{i\neq k}A_i\leq y_n) & =\mathbf{P}\left(N_n\geq x_n, \max_{i\neq k}A_i\leq y_n, Z_k-1>y_n, X_k=1\right)\\
                                                            & +\mathbf{P}\left(N_n\geq x_n, \max_{i\neq k}A_i\leq y_n, Z_k-1<-y_n, X_k=-1\right).
\end{align*}
For the first term in the last bound, we have the following decomposition:
\begin{align*}
& \mathbf{P}(N_n\geq x_n, \max_{i\neq k}A_i\leq y_n, Z_k-1>y_n, X_k=1)\\
& =\int_{-\infty}^{\infty}\mathbf{P}\left(Z_k-1>\max\{y_n, x_n-u\}\right) d\mathbf{P}\Big(N_n-A_k<u, \max_{i\neq k}A_i\leq y_n, X_k=1\Big)\\
& =\int_{-z_n}^{x_n'}+\int_{-\infty}^{-z_n}+\int_{x_n ' }^{\infty}:= I_1+I_2+I_3.
\end{align*}
We choose $x_n', z_n\rightarrow\infty$ such that
\begin{equation}\label{xn}
x_n/x_n ' \rightarrow\infty, \quad y_n/x_n ' \rightarrow0\quad \textrm{and} \quad x_n ' /\sqrt{n\ln n}\rightarrow\infty
\end{equation}
and
\begin{equation}\label{zn}
z_n/\sqrt{n}\rightarrow\infty, \quad x_n/z_n\rightarrow\infty.
\end{equation}
In the sequel, we give some estimations for the terms  $I_1$, $I_2$ and $I_3.$

\noindent\textit{a)} \textit{Control of $I_1$}.
By \eqref{xn}, \eqref{zn} and the condition \eqref{fcdnsa}, it is easy to see that \textcolor{black}{
\begin{equation*}
\frac{  \mathbf{P}(Z_1-1>x_n+z_n) }{\mathbf{P}(Z_1-1>x_n)}\rightarrow1
\end{equation*}
and
\begin{equation*}
\frac{  \mathbf{P}(Z_1-1>x_n-x_n ' )}{\mathbf{P}(Z_1-1>x_n)} \rightarrow1
\end{equation*}
as $n\rightarrow\infty.$}
Thus, it holds
\begin{equation*}
\lim_{n\rightarrow\infty}\sup_{-z_n\leq u\leq x_n ' } \bigg| \frac{\mathbf{P}(Z_1-1>x_n-u)}{\mathbf{P}(Z_1-1>x_n)}-1 \bigg|=0.
\end{equation*}
For sufficiently large $n$ and $-z_n\leq u\leq x_n ' $, we have
\begin{equation*}
\mathbf{P}(Z_1-1\geq \max\{y_n, x_n-u\})=\mathbf{P}(Z_1-1\geq x_n-u).
\end{equation*}
Therefore, it holds
\begin{equation}\label{I1.0}
I_1=\mathbf{P}(Z_1-1>x_n)\mathbf{P}\Big(-z_n<N_n-A_k<x_n^{'}, \max_{i\neq k}A_i\leq y_n, X_k=1\Big)\Big(1+o(1) \Big).
\end{equation}
For $\mathbf{P}\left(-z_n<N_n-A_k<x_n' , \max_{i\neq k}A_i\leq y_n, X_k=1\right)$, we have
\begin{align*}
&\mathbf{P}\left(-z_n<N_n-A_k<x_n', \max_{i\neq k}A_i\leq y_n, X_k=1\right)\\
&=\mathbf{P}\left(\max_{i\neq k}A_i\leq y_n, X_k=1\right) -\mathbf{P}\left(N_n-A_k\leq -z_n, \max_{i\neq k}A_i\leq y_n,X_k=1\right)\\
&\ \ \   -\mathbf{P}\left(N_n-A_k\geq x_n', \max_{i\neq k}A_i\leq y_n,X_k=1\right).
\end{align*}
For $\mathbf{P}\left(\max_{i\neq k}A_i\leq y_n, X_k=1\right)$, it holds
\begin{align*}
\frac{1}{2}\geq\mathbf{P}\left(\max_{i\neq k}A_i\leq y_n, X_k=1\right) &=\mathbf{P}\left(\max_{i\neq k}A_i\leq y_n\bigg| X_k=1\right)\mathbf{P}(X_k=1)\\
&=\frac{1}{2}\bigg[1-\mathbf{P}\bigg(\bigcup_{i\neq k}\{A_i>y_n\}\bigg|X_k=1\bigg)\ \bigg]\\
&\geq\frac{1}{2}\bigg[1-2\mathbf{P}\bigg(\bigcup_{i\neq k}\{A_i>y_n\}\bigg)\bigg].
\end{align*}
For $\mathbf{P}\big(\bigcup_{i\neq k}\{A_i>y_n\}\big)$, we can deduce that
\begin{eqnarray*}
\mathbf{P}\bigg(\bigcup_{i\neq k}\{A_i>y_n\}\bigg) &\leq& \sum_{i\neq k}\mathbf{P}(A_i>y_n)\leq \sum_{i\neq k}\mathbf{P}(|A_i|>y_n) \\
&\leq& n\mathbf{P}( |Z_1-1| >y_n)\\
&=& o(1).
\end{eqnarray*}
Thus, it holds
\begin{equation*}
\mathbf{P}\Big(\max_{i\neq k}A_i\leq y_n, X_k=1\Big)=\frac{1}{2}\Big(1+o(1) \Big).
\end{equation*}
It is easy to see that $$\mathbf{P}\left(N_n-A_k\leq -z_n, \max_{i\neq k}A_i\leq y_n,X_k=1\right)=o(1)$$ and  $$\mathbf{P}\left(N_n-A_k\geq x_n', \max_{i\neq k}A_i\leq y_n,X_k=1\right)=o(1)$$ as $n\rightarrow\infty.$
Therefore, from \eqref{I1.0}, the following equality holds
\begin{equation}\label{I1}
I_1=\frac{1}{2}\mathbf{P}(Z_1-1>x_n)\big(1+o(1) \big).
\end{equation}

\noindent\textit{b)}\textit{Control of $I_2$}.
By Chebyshev's inequality and (\ref{zn}), we get
\begin{equation*}
\mathbf{P}\left(N_n-A_k<-z_n, \max_{i\neq k}A_i\leq y_n\right)\leq \mathbf{P}(N_n-A_k<-z_n)\leq  \frac{(n-1)\sigma^2}{z_n^2} =o(1),
\end{equation*}
which implies
\begin{align}\label{I2}
I_2 &\leq \mathbf{P}(Z_k-1>x_n+z_n)\mathbf{P}\big(N_n-A_k<-z_n, \max_{i\neq k}A_i\leq y_n\big)\notag\\
    &\leq \mathbf{P}(Z_1-1>x_n)\mathbf{P}\big(N_n-A_k<-z_n, \max_{i\neq k}A_i\leq y_n\big)=o\Big(\mathbf{P}( Z_1-1 >x_n)\Big).
\end{align}

\noindent\textit{c)}\textit{Control of $I_3$}.
For $I_3$, we have \textcolor{black}{
\begin{align}\label{I3leq}
I_3 &\leq \int_{x_n'}^{\infty}\mathbf{P}(Z_k-1>y_n)d\mathbf{P}\left(N_n-A_k<u, \max_{i\neq k}A_i\leq y_n, X_k=1\right)\notag\\
    &\leq \mathbf{P}(N_n-A_k >x_n', \max_{i\neq k}A_i\leq y_n, X_k=1 )\mathbf{P}( Z_k-1 >y_n) \notag\\
    &\leq \mathbf{P}(N_n-A_k>x_n')\mathbf{P}( Z_k-1 >y_n)    \notag\\
    &\sim \mathbf{P}(N_n-A_k>x_n')\frac{\theta\mathcal{L}(y_n)}{y_n^{\gamma}}.
\end{align}
It can be easily deduced that $\mathbf{E} |Z_k-1|^{\gamma-\omega} <\infty$, where $\omega=(\gamma-2)/2$. By Corollary 2.5 in \cite{FGLDIMA}, it holds
\begin{equation}\label{NnI3}
\mathbf{P}(N_n-A_k>x_n')\leq \exp\left\{-\frac{x_n'^{2}}{2V^2}\right\}+\frac{B_{\gamma-\omega}}{x_n'^{\gamma-\omega}},
\end{equation}
where
\begin{equation}\label{Vc}
V^2=\frac{1}{4}(\gamma-\omega+2)^2e^{\gamma-\omega}n\sigma^2\quad and \quad B_{\gamma-\omega}=\Big(1+\frac{2}{\gamma-\omega} \Big)^{\gamma-\omega}nc_1.
\end{equation}
It is easy to see that
\begin{equation}\label{expxn}
\frac{\mathcal{L}(y_n)}{y_n^{\gamma}}\exp\left\{-\frac{x_n'^{2}}{2V^2}\right\}=o\Big(\mathbf{P}( Z_1-1 >x_n)\Big)
\end{equation}
and
\begin{equation}\label{powerxn}
\frac{\mathcal{L}(y_n)}{y_n^{\gamma}}\cdot\frac{B_{\gamma-\omega}}{x_n'^{\gamma-\omega}}
=O\Big(\frac{n\lambda_n^{\gamma}}{x_n'^{\gamma-\omega}}\cdot\frac{\mathcal{L}(x_n)}{x_n^\gamma} \Big)
=o\Big(\mathbf{P}( Z_1-1 >x_n) \Big).
\end{equation}
Therefore, it follows by \eqref{I3leq}, \eqref{NnI3}, \eqref{expxn} and \eqref{powerxn} that
\begin{equation}\label{I3}
I_3=o\Big(\mathbf{P}( Z_1-1 >x_n) \Big).
\end{equation} }
Combining the equalities  \eqref{I1}, \eqref{I2} and \eqref{I3} together, we obtain
\begin{equation*}
\sum_{k=1}^{n}\mathbf{P}\left(N_n\geq x_n, \max_{i\neq k}A_i\leq y_n, Z_k-1>y_n, X_k=1\right)=\frac{n}{2}\mathbf{P}(Z_1-1>x_n)\big(1+o(1)\big).
\end{equation*}
For $\mathbf{P}(N_n\geq x_n, \max_{i\neq k}A_i\leq y_n, Z_k-1<-y_n, X_k=-1)$, we have
\begin{align*}
&\quad \mathbf{P}(N_n\geq x_n, \max_{i\neq k}A_i\leq y_n, Z_k-1<-y_n, X_k=-1)\\
& =\int_{-\infty}^{\infty}\mathbf{P}\left(Z_k-1<\min\{-y_n, -x_n+u\}\right) d\mathbf{P}\big(N_n-A_k<u, \max_{i\neq k}A_i\leq y_n, X_k=-1\big).
\end{align*}
Similarly, we can deduce that
\begin{eqnarray}
&&\sum_{k=1}^{n}\mathbf{P}\Big(N_n\geq x_n, \max_{i\neq k}A_i\leq y_n, Z_k-1<-y_n, X_k=-1\Big) \nonumber \\
&& \ \ \ \ \ \ \ \ \ \ \ \ \ \ \ \ \ \  \ \ \ \ \ \ \ \ \  \ \ \ \ \ \ \ \ \  = \frac{n}{2}\mathbf{P}(Z_1-1<-x_n)\big(1+o(1)\big).\label{Xk-1}
\end{eqnarray}
Combining \eqref{I1} and \eqref{Xk-1} together, it holds
\begin{equation}\label{P1n}
P_{1n}=\frac{n}{2}\mathbf{P}(|Z_1-1|>x_n)\big(1+o(1) \big).
\end{equation}

\textbf{(II)} \textbf{Control of $P_{2n}$}.
It is easy to see that
\begin{equation*}
P_{2n} \leq \mathbf{P}\bigg(\bigcup_{1\leq i<j\leq n}\{A_i>y_n, A_j>y_n\}\bigg)\leq \sum_{1\leq i<j\leq n}\mathbf{P}(A_i>y_n, A_j>y_n).
\end{equation*}
For $\mathbf{P}(A_i>y_n, A_j>y_n)$, it holds
\begin{align*}
\mathbf{P}(A_i>y_n, A_j>y_n)\leq\mathbf{P}(|A_i|>y_n, |A_j|>y_n)=\mathbf{P}^2(|Z_1-1|>y_n).
\end{align*}
When $n$ is large enough, we have
\begin{align*}
P_{2n} & \leq n^2\mathbf{P}^2(|Z_1-1|>y_n)\sim n^2\mathcal{L}^2(y_n)y_n^{-2\gamma}=\mathcal{L}^2(y_n)n^2\lambda_n^{2\gamma}/x_n^{2\gamma}\\
                & \leq n^{1-\gamma/2}(n\mathcal{L}(x_n)/x_n^\gamma),
\end{align*}
which implies
\begin{equation}\label{P2n}
P_{2n}=o\big(n\mathbf{P}(|Z_1-1|>x_n)\big).
\end{equation}

\textbf{(III)} \textbf{Control of $P_{3n}$}.
By Theorem 2.4 in \cite{FGLDIMA} with $x=x_n$ and $y=y_n$, then the following inequality holds
\begin{align}
P_{3n} &= \mathbf{P}\left(N_n\geq x_n, \sum_{k=1}^{n}\mathbf{E}(A_k^2|\mathcal{F}_{k-1})=n\sigma^2, \sum_{k=1}^{n}\mathbf{E}(|A_k|^{\gamma-\omega}|\mathcal{F}_{k-1})= c_2n, \max_{1\leq k\leq n}A_k\leq y_n \right) \nonumber \\
   &\leq \exp\left\{-\frac{x_n^2}{2V^2}\right\}+\left(\frac{c_2n\lambda_n^{\gamma-\omega-1}}{x_n^{\gamma-\omega}}\right)^{\lambda_n},\label{ine615}
\end{align}
where $V^2$ and $\omega$ is defined by \eqref{NnI3} and \eqref{Vc}.
For the first term in the bound (\ref{ine615}), it is easy to see that
\begin{equation}\label{P3n1}
\exp\left\{-\frac{x_n^2}{2V^2}\right\}=o\Big(n\mathbf{P}(|Z_1-1|\geq x_n) \Big).
\end{equation}
And for the second term in the bound (\ref{ine615}), it holds
\begin{equation*}
\left(\frac{c_2n\lambda_n^{\gamma-\omega-1}}{x_n^{\gamma-\omega}}\right)^{\lambda_n}\Bigg/\frac{n\mathcal{L}(x_n)}{x_n^\gamma}
\leq\frac{c_2^{\lambda_n}x_n^{2\lambda_n}\lambda_n^{\gamma\lambda_n/2}}{\mathcal{L}(x_n)x_n^{(\gamma+2)\lambda_n/2}}
\leq \frac{c\lambda_n^{\gamma\lambda_n/2}}{\mathcal{L}(x_n)(x_n/c_3)^{(\gamma-2)\lambda_n/2}}=o(1),
\end{equation*}
which means
\begin{equation}\label{P3n2}
\left(\frac{c_2n\lambda_n^{\gamma-\omega-1}}{x_n^{\gamma-\omega}}\right)^{\lambda_n}=o\Big(n\mathbf{P}(|Z_1-1|\geq x_n)\Big).
\end{equation}
Combining \eqref{P3n1} and \eqref{P3n2} together, we get
\begin{equation}\label{P3n}
P_{3n}=o\Big(n\mathbf{P}(|Z_1-1|\geq x_n) \Big).
\end{equation}

Combining  \eqref{P1n}, \eqref{P2n} and \eqref{P3n} together, we deduce that
\begin{equation}\label{Nn1}
\mathbf{P}(N_n\geq x_n)=\frac{n}{2}\mathbf{P}(|Z_1-1|\geq x_n)\big(1+o(1) \big).
\end{equation}
By the condition \eqref{fcdnsa}, we have
\begin{equation*}
\mathbf{P}(|Z_1-1|\geq x_n)/\mathbf{P}(|Z_1|\geq x_n)\rightarrow1 \quad \textrm{as} \  n\rightarrow\infty,
\end{equation*}
which together with (\ref{Nn1}) implies that
\begin{equation}\label{Nn}
\mathbf{P}(N_n\geq x_n)=\frac{n}{2}\mathbf{P}(|Z_1|\geq x_n)\big(1+o(1) \big).
\end{equation}

For any $x_n\gg \sqrt{n\ln n}$, there exists a  function $\mathcal{G}$ such that $\mathcal{G}(n)\rightarrow\infty$ as $n\rightarrow\infty$ and $x_n=\sqrt{n\ln n} \mathcal{G}(n)$. Set $\Delta_n=x_n/\sqrt{\mathcal{G}(n)}=o(x_n)$. Clearly, it holds
\begin{equation}\label{Snup1}
\mathbf{P}(S_n\geq x_n)\leq \mathbf{P}(N_n\geq x_n-\Delta_n)+\mathbf{P}(T_n\geq\Delta_n).
\end{equation}
By (\ref{Nn1}) and (\ref{fcdnsa}), we can deduce  that
\begin{equation}\label{xn-deltan}
\mathbf{P}(N_n\geq x_n-\Delta_n)/\mathbf{P}(N_n\geq x_n)\rightarrow1 \ \ \ \ \textrm{as} \ n\rightarrow\infty.
\end{equation}
By (3.18) and (3.22) in \cite{HFDI}, we have the following  Azuma-Hoeffding type inequality for the ERW:
\begin{align}\nonumber
\mathbf{P}(T_n\geq \Delta_n)\leq \exp\left\{-c_5\Delta_n^2/n\right\}=\exp\left\{-c_5x_n^2/(n\mathcal{G}(n))\right\}, \quad p\in[0, 3/4).
\end{align}
By the last inequality, we can deduce that
\begin{align}\label{TnDeltan}
\mathbf{P}(T_n\geq \Delta_n)  =o\big(\mathbf{P}(N_n\geq x_n) \big).
\end{align}
Combining \eqref{Snup1}, \eqref{xn-deltan} and \eqref{TnDeltan}, we get
\begin{equation}\label{upper1}
\mathbf{P}(S_n\geq x_n)\leq \mathbf{P}(N_n\geq x_n)\big(1+o(1)\big).
\end{equation}

It is easy to see that
\begin{align}\label{Snlow2}
\mathbf{P}(S_n\geq x_n) &\geq \mathbf{P}(N_n\geq x_n+\Delta_n, T_n\geq-\Delta_n)\notag\\
                        &=\mathbf{P}(N_n\geq x_n+\Delta_n)-\mathbf{P}(N_n\geq x_n+\Delta_n, T_n<-\Delta_n)\notag\\
                        &\geq\mathbf{P}(N_n\geq x_n+\Delta_n)-\mathbf{P}(T_n<-\Delta_n).
\end{align}
Similarly, it can be deduced that
\begin{equation}\label{xn-deltan2}
\mathbf{P}(N_n\geq x_n+\Delta_n)/\mathbf{P}(N_n\geq x_n)\rightarrow1 \quad \textrm{as} \quad n\rightarrow\infty
\end{equation}
and
\begin{equation}\label{TnDeltan2}
\mathbf{P}(T_n<-\Delta_n)=o\big(\mathbf{P}(N_n\geq x_n)\big).
\end{equation}
Thus, by \eqref{Snlow2}, \eqref{xn-deltan2} and \eqref{TnDeltan2}, we have the following  lower bound
\begin{equation}\label{lower1}
\mathbf{P}(S_n\geq x_n)\geq \mathbf{P}(N_n\geq x_n)\big(1+o(1)\big).
\end{equation}

Combining  \eqref{upper1} and \eqref{lower1} together, we get
 \begin{equation} \nonumber
\mathbf{P}(S_n\geq x_n)=\mathbf{P}(N_n\geq x_n)\big(1+o(1)\big).
\end{equation}
The last equality  and \eqref{Nn} together implies the first desired equality.
 The second desired  equality follows by a similar argument. This completes the proof of Theorem \ref{thm5}.

\section*{Acknowledgements}
The authors are deeply indebted to the editor and two anonymous referees for their helpful
comments. This work has been partially supported  by  the National Natural Science Foundation of China (Grant No.\,12371155) and  Natural Science Foundation of Hebei Province (Grant No.\,A2025501005).


\begin{thebibliography}{00}{\footnotesize
\bibitem{ERWPTU} Baur E, Bertoin J. Elephant random walks and their connection to P\'{o}lya-type urns.  Phys Rev E, 2016, 94(5), 052134.
\bibitem{BMERW}Bercu  B. A martingale approach for the elephant random walk.  J  Phys  A: Math Theor, 2017, 51(1), 015201.

\bibitem{BMERWS}Bercu B. On the multidimensional elephant random walk with stops.  Stoch  Process  Their Appl, 2025,  189, 104692.
\bibitem{BLMERW}Bercu  B, Laulin  L. On the multi-dimensional elephant random walk.  J Stat Phys, 2019, 175: 1146--1163.
\bibitem{B21} Bertoin  J. Counting the zeros of an elephant random walk.  Trans  Amer  Math  Soc, 2022,  375: 5539--5560.
\bibitem{CGSCLT} Coletti  C F, Gava R, Sch{\"u}tz G M. Central limit theorem and related results for the elephant random walk.  J  Math  Phys, 2017, 58(5), 053303.
\bibitem{ASIP}Coletti C  F, Gava R,  Sch{\"u}tz G M. A strong invariance principle for the elephant random walk.  J. Stat. Mech.  2017, 2017(12), 123207.
\bibitem{DFHRCCLT}Dedecker  J, Fan X, Hu H,  Merlev\`{e}de F.  Rates of convergence in the central limit theorem for the elephant random walk with random step sizes.  J Stat Phys, 2023, 190(10), 154.
\bibitem{DDSLDP}Denisov  D, Dieker A B,  Shneer V. Large deviations for random walks under subexponentiality: the big-jump domain. Ann  Probab,  2008, 36: 1946--1941.
\bibitem{FGLDIMA}Fan  X, Grama I,  Liu Q. Deviation inequalities for martingales with applications.  J  Math Anal  Appl,  2017, 448(1): 538--566.
\bibitem{FSCMD}Fan  X, Shao Q M. Cram\'{e}r's moderate deviations for martingales with applications.  Ann  Inst  H Poincar\'{e} Probab  Statist, 2024. 60(3): 2046--2074.
\bibitem{GHLDP}Grama  I, Haeusler E. Large deviations for martingales via Cram\'{e}r's method.  Stochastic Process  Appl, 2000, 85(2): 279--293.
\bibitem{GSVERW}Gut  A,  Stadtm{\"u}ller U. Variations of the elephant random walk.  J Appl Probab, 2021, 58(3): 805--829.
\bibitem{HFDI} Hu H,  Fan X. Deviation inequalities for the elephant random walk with random step sizes.  J  Math  Phys, 2025, 66(3), 033301
\bibitem{LLAERW}Laulin  L. Introducing smooth amnesia to the memory of the elephant random walk.  Electron  Commun  Probab, 2022, 27: 1--12.
\bibitem{MWPLDPRVS}Mikosch  T, Wintenberger O. Precise large deviations for dependent regularly varying sequences.  Probab. Theory Related Fields, 2013, 156(3): 851--887.
\bibitem{NAVILTLDP} Nagaev  A V. Integral limit theorems for large deviations when Cram\'{e}r's condition is not fulfilled I, II.  Theory Probab  Appl, 1969,  14: 51--64 and 193--208.
\bibitem{NSVLDP} Nagaev  S  V. Large deviations of sums of independent random variables.  Ann  Probab, 1979, 7(5): 745--789.
\bibitem{N} Nakano  Y. Elephant random walk with polynomially decaying steps.  J  Stat  Phys, 2025, 192(6), 86.
\bibitem{QS} Qin  S. Recurrence and transience of multidimensional elephant random walks.  Ann  Probab,  2025, 53(3): 1049--1078.
\bibitem{schutz2004elephants} Sch{\"u}tz   G  M, Trimper S. Elephants can always remember: Exact long-range memory effects in a non-Markovian random walk.  Phys  Rev  E, 2004, 70(4), 045101.
}
\end{thebibliography}
\end{document}